\PassOptionsToPackage{hidelinks,hypertexnames=false}{hyperref}
\documentclass[pdflatex,sn-mathphys-num]{sn-jnl}
\usepackage{graphicx}
\usepackage{amsmath,amssymb,amsfonts,amsthm}
\usepackage{mathtools}
\usepackage[title]{appendix}
\usepackage{xcolor}
\usepackage{booktabs,multirow}
\usepackage{algorithm,algorithmic}
\usepackage{enumerate}
\usepackage{xargs}

\usepackage{placeins}

\theoremstyle{thmstyleone}
\newtheorem{theorem}{Theorem}
\newtheorem{proposition}{Proposition}

\theoremstyle{thmstyletwo}
\newtheorem{remark}{Remark}

\theoremstyle{thmstylethree}
\newtheorem{definition}{Definition}
\newtheorem{assumption}{Assumption}

\newif\ifrevision
\revisionfalse
\newenvironment{revision}{\par\begingroup\ifrevision\color{red}\fi}{\par\endgroup}

\newcommand{\bluechange}[1]{{\ifrevision\color{blue}\fi #1}}

\def\R{\mathbb{R}}

\DeclareMathOperator{\diag}{diag}

\DeclareMathOperator{\spn}{span}
\DeclareMathOperator{\diver}{div}

\def\div{\nabla\cdot}

\newcommand{\bw}{\mathbf{w}}

\newcommand{\by}{\mathbf{y}}
\newcommand{\bu}{\mathbf{u}}

\newcommand{\bh}{\mathbf{h}}

\newcommand{\bv}{\mathbf{v}}

\begin{document}
\title[Efficient RHC for 2D Navier--Stokes Flows]{Efficient Receding Horizon Control for Stabilizing 2D Navier--Stokes Flows with Finitely Many Localized Actuators and POD Acceleration}
\author*[1]{\fnm{Behzad} \sur{Azmi}}\email{behzad.azmi@uni-konstanz.de}
\author[1]{\fnm{Stefan} \sur{Frei}}\email{stefan.frei@uni-konstanz.de}
\author[2]{\fnm{Felix} \sur{Sauer}}\email{sauer@wias-berlin.de}
\affil*[1]{\orgdiv{Department of Mathematics and Statistics}, \orgname{University of Konstanz}, \orgaddress{\street{Universit\"atsstra\ss e 10}, \city{Konstanz}, \postcode{78457}, \country{Germany}}}
\affil[2]{\orgname{Weierstrass Institute for Applied Analysis and Stochastics (WIAS)}, \orgaddress{\street{Anton-Wilhelm-Amo-Str.\ 39}, \city{Berlin}, \postcode{10117}, \country{Germany}}}
\abstract{This paper develops and analyzes a computationally efficient receding horizon control (RHC) strategy for steering two-dimensional Navier--Stokes flows to prescribed reference trajectories. The control is represented by a finite linear combination of localized actuators supported within a prescribed control subdomain. At the continuous PDE level, we establish local exponential stabilization in the $L^2$-norm and suboptimality of the actuator-based RHC scheme for weak solutions. Direct implementation requires the repeated solution of high-dimensional finite-horizon optimal control problems. To reduce this computational cost, we combine RHC with model order reduction based on proper orthogonal decomposition (POD). A common POD space constructed from state snapshots reduces the state and adjoint equations and accelerates the online optimization, while the resulting controls are applied to the full-order discretized dynamics. Numerical experiments for two flow configurations of increasing complexity examine the influence of actuator layouts, prediction horizons, viscosity, and POD dimension. The numerical results demonstrate substantial reductions in computational cost while maintaining closed-loop stabilization in the tested configurations.}
\keywords{receding horizon control, 2D Navier--Stokes equations, finite-dimensional control, stabilization, model order reduction, proper orthogonal decomposition}
\pacs[MSC Classification]{93C20, 35Q30, 49J20, 93D20}
\maketitle

\section{Introduction}
In this paper,  we investigate the stabilizability of the following controlled Navier-Stokes system 
\begin{equation}\label{e1}
\begin{aligned}
&\partial_t \mathbf{y}-\nu\Delta \mathbf{y} + ( \mathbf{y}\cdot\nabla)\mathbf{y} + \nabla \mathbf{p}
= \sum_{i=1}^N u_i\,\mathbf{1}_{\mathbf{R}_i} + \hat{\mathbf{f}}
\quad \text{ in } (0,\infty)\times\Omega,\\
&\nabla\cdot \mathbf{y}=0,\quad \mathbf{y}=\hat{\mathbf{h}}\ \text{ on }(0,\infty)\times\partial\Omega,\quad
\mathbf{y}(0,\cdot)=\mathbf{y}_0\ \text{in }\Omega,
\end{aligned}
\end{equation}
around a given trajectory $\hat{ \mathbf{y}}$ via receding horizon control (RHC) from both theoretical and computational perspectives. Here $\Omega \subset \mathbb{R}^2$ is a bounded domain with a smooth boundary $\partial \Omega$. The vector-valued function $\mathbf{y}(t,x) = (y_1(t,x), y_2(t,x))$ represents fluid velocities, the real-valued function $\mathbf{p}(t,x)$ indicates the pressure field, and $\hat{\mathbf{f}}(t,x) = (\hat{f}_1(t,x), \hat{f}_2(t,x))$ and   $\hat{ \mathbf{h}}= (\hat{h}_1(t,x), \hat{h}_2(t,x))$ are the distributed and boundary source fields. Moreover, $\nu>0$ denotes the viscosity constant, and \(\mathbf y\cdot\nabla\) stands for the differential operator \(y_1\partial_{x_1}+y_2\partial_{x_2}\). For $N\in\mathbb N$, the functions $\mathbf1_{\mathbf R_i}$,
$i=1,\ldots,N$, are prescribed vector-valued actuator profiles
based on indicator functions, each acting on one velocity component.
We assume that
$\operatorname{supp}(\mathbf1_{\mathbf R_i})\subseteq\omega$
for all $i=1,\ldots,N$, where $\omega\subseteq\Omega$ is a
prescribed open control region. The actuator construction is
specified in Section~\ref{sec:stabilization}. 
 
 We aim to find an efficiently computable control   $\mathbf{u}(t):=[ u_1(t), \dots, u_N(t) ] \in$ $L^2((0,\infty);\mathbb{R}^N)$ using the receding horizon (RH) framework to steer system \eqref{e1} to a  reference  trajectory $\hat{ \mathbf{y}}$ satisfying
\begin{equation}\label{e2}
\begin{aligned}
&\partial_t\hat{ \mathbf{y}}-\nu\Delta \hat{ \mathbf{y}} +  (\hat{ \mathbf{y}} \cdot \nabla  ) \hat{ \mathbf{y}} + \nabla \hat{ \mathbf{p}} = \hat{ \mathbf{f}}
\quad \text{ in } (0,\infty)\times\Omega,\\
&\nabla\cdot \hat{ \mathbf{y}}=0,\quad \hat{ \mathbf{y}}=\hat{\mathbf{h}}\ \text{ on }(0,\infty)\times\partial\Omega,\quad
\hat{ \mathbf{y}}(0,\cdot)=\hat{\mathbf{y}}_0\ \text{in }\Omega,
\end{aligned}
\end{equation}
for initial states $\mathbf y_0$ in a sufficiently small $L^2$-neighborhood of the fixed reference initial state $\hat{\mathbf y}_0$. Here,  $\hat{\mathbf{y}}(t,x)=( \hat{y}_1(t,x),\hat{y}_2(t,x))$ and  $\hat{\mathbf{p}}(t,x)$ denote the associated fluid velocities and the pressure field, respectively.  More precisely,  we will show that there exists $d>0$ such that for every initial function $\mathbf{y}_0$ satisfying 
$\|\mathbf{y}_0-\hat{\mathbf{y}}_0\|_{L^2(\Omega;\mathbb{R}^2)}  \leq d$,  and the RH state $\mathbf{y}_{rh}$ corresponding to the RHC  $\mathbf{u}_{rh}(\mathbf{y}_0) \in  L^2((0,\infty);\mathbb{R}^N)$,  it holds
\begin{equation}
\label{main-stability-result}
\| \mathbf{y}_{rh}(t)-\hat{\mathbf{y}}(t) \|^2_{L^2(\Omega;\mathbb{R}^2)} \leq c_H e^{-\zeta t} \| \mathbf{y}_0- \hat{\mathbf{y}}_0\|^2_{L^2(\Omega;\mathbb{R}^2)}    \quad  \forall t>0,    \end{equation} 
where  the positive constants $c_H$ and $\zeta$ are independent of $\mathbf{y}_0$.  

Recently,  the flexibility and efficiency of the RH framework have attracted attention in the field of control systems governed by partial differential equations (PDEs). This approach formulates stabilizing control as a solution to an infinite-horizon optimal control problem,  whose performance index function enhances desirable properties and structures on the control. For every $T \in (0,\infty]$ and every initial pair
$(\bar{t}_0,\bar{\by}_0)\in \mathbb{R}_+\times L^2(\Omega;\mathbb{R}^2)$, we define the performance index
\begin{equation}
\label{e49}
J_T^o(\bu;\bar{t}_0,\bar{\by}_0)
:=
\frac{1}{2}\int_{\bar{t}_0}^{\bar{t}_0+T}
\|\by(t)-\hat{\by}(t)\|_{L^2(\Omega;\mathbb{R}^2)}^2\,dt
+\frac{\beta}{2}\int_{\bar{t}_0}^{\bar{t}_0+T}
|\bu(t)|_2^2\,dt,
\end{equation}
where $\beta>0$ and $|\cdot|_2$ denotes the $\ell_2$-norm. Then, the infinite horizon optimal control problem is formulated as 
\begin{equation}
\inf_{\mathbf{u} \in L^{2}((0,\infty); \mathbb{R}^N) }\{ J^o_{\infty}(\mathbf{u};0,\mathbf{y}_0) \text{ s.t. } \eqref{e1} \text{ and } \eqref{e2} \}.
\tag*{\ensuremath{\mathrm{OP}^o_{\infty}(\mathbf{y}_0)}}
\label{opinf}
\end{equation}
Subsequently, the solution of \eqref{opinf} is approximated by concatenating a sequence of finite-horizon open-loop optimal controls on overlapping time intervals covering $[0,\infty)$. These controls are obtained as solutions of problems of the form
\begin{equation}
\label{opT}
\tag*{\ensuremath{\mathrm{OP}^o_{T}(\bar{t}_0, \bar{\mathbf{y}}_0)}}
\min_{\mathbf{u} \in L^{2}((\bar{t}_0,\bar{t}_0+T); \mathbb{R}^N) }\{ J^o_{T}(\mathbf{u};\bar{t}_0, \bar{\mathbf{y}}_0) \text{ s.t. } \eqref{e1} \text{ and } \eqref{e2} \}.
\end{equation}
 The RH framework provides a bridge between open-loop and closed-loop control. A central issue is to establish the stabilizability of the induced feedback law. This is typically achieved by suitable concatenation arguments or by augmenting the finite-horizon subproblems with terminal costs and/or terminal constraints. In this work, we adopt the RH framework of \cite{Azmi5}, where an appropriate concatenation scheme ensures stabilizability, and neither terminal costs nor terminal constraints are required. Related approaches have been investigated for continuous-time ODEs \cite{RA12} and discrete-time systems \cite{G09,GR08}.
  The resulting procedure is summarized in Algorithm~\ref{RHA}; here, $\delta$ denotes the sampling time, while $\by_{rh}$ and $\bu_{rh}$ denote the RH state and control, respectively. Specifically, we show the following for Algorithm \ref{RHA}: 
  
\begin{algorithm}[htbp]
\caption{RHC($\delta,T$)}\label{RHA}
\begin{algorithmic}[1]
\REQUIRE{Final time $T_{\infty} \in \mathbb{R}_{\geq 0} \cup \{ \infty\}$, sampling time $\delta$, the prediction horizon $T\geq \delta$, 
a reference trajectory $\hat{\mathbf{y}}$, and the initial state $\mathbf{y}_0$.}
\ENSURE{The stabilizing RHC~$\mathbf{u}_{rh}$, nondecreasing sequence~$\{ t_i \}_{i\in\mathbb{N}}$.}
\STATE Set $\bar t_0 \gets 0$, $\bar\by_0 \gets \by_0$, and $i\gets 0$.
\WHILE{$ {\bar{t}_{0}} <T_\infty$}
\STATE Find the optimal solution $(\mathbf{y}_T^*(\cdot;\bar{t}_0, \bar{\mathbf{y}}_0), \mathbf{u}^*_T(\cdot;\bar{t}_0, \bar{\mathbf{y}}_0))$
over the time horizon $(\bar{t}_0,\bar{t}_0+T)$ by solving the open-loop problem~\eqref{opT}.
\STATE Set $t_{i+1}\gets \min\{\bar t_0+\delta,T_\infty\}$.
\STATE For all $\tau\in[\bar{t}_0,t_{i+1})$ set $\mathbf{y}_{rh}(\tau):= \mathbf{y}_T^*(\tau;\bar{t}_0,\bar{\mathbf{y}}_0)$ and $\mathbf{u}_{rh}(\tau):=\mathbf{u}^*_T(\tau;\bar{t}_0,\bar{\mathbf{y}}_0)$.
 \STATE Update $\bar\by_0 \gets \by^{*}_T(t_{i+1};\bar t_0,\bar\by_0)$, $\bar t_0\gets t_{i+1}$, and $i\gets i+1$.
\ENDWHILE
\end{algorithmic}
\end{algorithm}

\noindent\textbf{Theoretical result.} For a sufficiently large number of actuators $N$ and a sufficiently large prediction horizon
$T \geq \delta$, Algorithm~\ref{RHA}  generates a finite-dimensional actuator-based receding horizon control $\mathbf{u}_{rh}$ that is locally exponentially stabilizing and suboptimal at the continuous PDE level with respect to $L^2(\Omega; \mathbb{R}^2)$.

We then present a series of numerical examples under different parameter settings to illustrate and support the theoretical results. Since the RH strategy requires repeatedly solving open-loop finite-horizon optimal control problems at intervals $(t_i,t_i+T)$, its direct numerical implementation can be computationally demanding. After discretization,  these subproblems are high-dimensional, so Algorithm~\ref{RHA} may become prohibitively expensive in practice.

To reduce this computational cost, we combine the RH framework with model order reduction (MOR) based on proper orthogonal decomposition (POD). We construct an initial POD basis from state snapshots generated during the full-order optimization on the first prediction horizon $(0,T)$ and use the same reduced space for the state and adjoint equations. We then use these reduced models to solve the open-loop subproblems on subsequent horizons $(t_i,t_i+T)$, $i\geq1$. If desired, we update the common basis online with new full-order state snapshots. We apply the computed controls to the full-order system over each sampling interval $[t_i,t_{i+1}]$, thereby obtaining the initial condition for the next RHC step. Numerical experiments demonstrate the computational efficiency and stabilization performance of the proposed MOR-based RHC scheme.

We emphasize that the stabilization and suboptimality results are established at the continuous PDE level for the finite-dimensional actuator-based RHC scheme, whereas the POD-MOR variant is developed as a computationally efficient implementation and assessed numerically.

\subsection{Related Work}
The numerical simulation of time-dependent PDEs, and in particular of the incompressible Navier--Stokes equations, is computationally demanding. This has motivated substantial research on MOR techniques, which approximate the underlying dynamics at significantly lower computational cost; see \cite{MR3419868} for a survey of projection-based MOR methods for parametric dynamical systems and \cite{rowley2017model} for an overview in the context of flow analysis and control.

Many widely used reduced-order models (ROMs) are snapshot-based in the sense that they are constructed from representative samples of the state trajectory. This perspective was introduced in the context of fluid dynamics in \cite{sirovich1987turbulence}. A prominent MOR approach is proper orthogonal decomposition (POD), in which a low-dimensional basis is extracted from snapshot data and then used to construct a ROM, typically by Galerkin projection; see, for example, \cite{gubisch2017proper,kunisch2002galerkin} for the POD--Galerkin framework and the corresponding error estimates. Beyond snapshot-based POD methods, system-theoretic MOR and robust control techniques have also been developed for incompressible flows. Moment-matching approaches for Navier--Stokes-type systems are studied in \cite{MR3708733}, while \cite{MR4405582} develops low-dimensional $\mathcal{H}_\infty$ controllers for stabilization of incompressible flows.

MOR techniques are particularly attractive in optimal control, where the repeated solution of high-dimensional state and adjoint equations often forms the main computational bottleneck. For POD-based MOR in optimal control, including error estimates and convergence results for linear-quadratic and semilinear parabolic problems, we refer to, e.g., \cite{MR3809542,hinze2008error,kammann2013posteriori,kunisch2008proper}. 

In the context of RHC, MOR is used to replace the full-order dynamics by reduced surrogate models in the repeated finite-horizon solves. Unlike in single-shot optimal control problems, however, reduction errors may accumulate over successive horizon shifts, making stability and performance guarantees for the closed-loop system generated by the reduced-order RHC scheme a central issue. For stability-preserving MOR--RHC schemes and related certification strategies, we refer, for instance, to \cite{alla2015asymptotic,azmi2024MORRHC,ghiglieri2014optimal,kartmann2024certified}. Certification approaches based on \emph{a posteriori} performance estimates for reduced RHC of parabolic PDEs have been developed in both discrete- and continuous-time settings; see \cite{dietze2025reduced} and \cite{azmi2025stabilization}, respectively. Adaptive and data-driven reduced modeling strategies have also been explored for feedback control and stabilization; see, for example, \cite{MR3686784,MR4610666,MR4760358}. For the theoretical design of stabilizing feedback laws for the Navier--Stokes equations, we refer, for example, to \cite{BreKunPfe19,lasiecka19,RodSei24} and the references therein.

\subsection{Contributions}
The main contributions of this manuscript are as follows:
\begin{itemize}
   
 \item We establish local exponential stabilization in the $L^2(\Omega;\mathbb R^2)$-norm and suboptimality of Algorithm~\ref{RHA} for weak solutions of the two-dimensional Navier--Stokes equations. This extends the strong-solution RHC analysis of \cite{Azmi5}, developed for the three-dimensional Navier--Stokes equations, to a weak-solution framework in two dimensions. The main difficulty concerns the spatial regularity of the nonlinear convection term: in the weak-solution framework, it can be estimated in $H^{-1}(\Omega;\mathbb R^2)$, whereas the spatial $L^2(\Omega;\mathbb R^2)$ estimate required by the semigroup perturbation argument is not available in general. Building on the finite-dimensional linear feedback construction of \cite{Azmi5}, we instead combine exponential stability with variational energy estimates to obtain an exponentially weighted estimate for forcing in $H^{-1}(\Omega;\mathbb R^2)$. Together with two-dimensional bilinear estimates and a fixed-point argument, this yields nonlinear stabilization for small initial perturbations in $L^2(\Omega;\mathbb R^2)$ and the cost bounds needed for the RHC suboptimality analysis.

 \item We present a systematic numerical study of Algorithm~\ref{RHA} on two examples of increasing complexity, covering different viscosities, domains, boundary conditions, prediction horizons, and actuator layouts.

    \item We develop a POD-based MOR implementation of the RHC scheme in which reduced state and adjoint models are used to solve the repeated finite-horizon optimal control problems. The reduced controls are applied to the full-order system in closed loop. The computational contribution is the comparison of fixed and updated common POD spaces for the state and adjoint within actuator-based Navier--Stokes RHC, including a case in which a fixed basis loses effectiveness. The experiments quantify reported runtime savings and illustrate the interaction between prediction horizon and reduced dimension.
         
\end{itemize}

\subsection{Organization of the Paper}
The remainder of the paper is organized as follows. In Section~\ref{sec:prelim}, we introduce the notation and function spaces for the Navier--Stokes equations and recall preliminary well-posedness and regularity results for the translated system. Section~\ref{sec:stabilization} is devoted to the stability and suboptimality analysis of the RHC scheme. In Section~\ref{sec:disc_and_impl}, we briefly describe the finite element and time discretization. Section~\ref{sec:MOR_RHC} introduces the POD-based MOR approach and the corresponding RHC algorithm. Numerical results are presented in Section~\ref{sec:Numerical_Results}, followed by concluding remarks in Section~\ref{sec:conclusions}. For readability, one proof from Section~\ref{sec:stabilization} is deferred to the appendix.

\section{Preliminaries}\label{sec:prelim}
\subsection{Notations and Functional Spaces}
We write $\mathbb R_+:=[0,\infty)$ and take finite prediction horizons $T>0$.   For a Banach space $X$, we denote by $\| \cdot \|_X$ the associated  norm, by $X'$ the associated dual space, and  by $\langle \cdot , \cdot  \rangle_{X',X}$ the dual pairing between $X'$ and $X$. In the case of a Hilbert space $X$, we use the scalar product $(\cdot,\cdot)_X$.  Further, $\mathcal{L}(X, Y )$ denotes the space of continuous linear operators from $X$ to $Y$ with the usual operator norm $\| \cdot\|_{\mathcal{L}(X,Y)}$. In case $X = Y$,  we write $\mathcal{L}(X) :=\mathcal{L}(X, X)$ instead.
Let  $X$ and $Y$ be Banach spaces, then for  any open interval $(t_0,t_1) \subset \mathbb{R}_+$ we define
\begin{equation*}
W((t_0,t_1); X,Y) := \{ \mathbf{y} \in L^2((t_0,t_1);X)   :   \partial_t \mathbf{y} \in L^2((t_0,t_1);Y)\},
\end{equation*}
where the derivative $\partial_t $ is taken in the sense of distributions. This space is endowed with the norm
\begin{equation*}
\|\mathbf{y}\|_{W((t_0,t_1);X,Y) } = \left( \|\mathbf{y}\|_{L^2((t_0,t_1);X) }^2  +\|\partial_t \mathbf{y}\|_{L^2((t_0,t_1);Y) }^2 \right)^{\frac{1}{2}}.
\end{equation*}
We also frequently use the notation $I_{T}(\bar{t}_0) :=(\bar{t}_0,\bar{t}_0+T)\subset \mathbb{R}_+$ with $\bar{t}_0 \in \mathbb{R}_+$  and $T \in (0,\infty]$. In this case, the closed intervals $[\bar{t}_0,\bar{t}_0+T]$ and $[\bar t_0, \infty)$ are denoted by $\overline{I}_{T}(\bar{t}_0)$ and  $\overline{I}_{\infty}(\bar{t}_0)$, respectively.     

Throughout, for simplicity, we use the notations $\mathbf{L}^p:= L^p(\Omega ;\mathbb{R}^2)$, $\mathbf{W}^{p,q}:= W^{p,q}(\Omega ;\mathbb{R}^2)$ for $p,q >0$,  $\mathbf{H}^1_0:= H^1_0(\Omega ;\mathbb{R}^2)$, and $\mathbf{H}^{-1} =(\mathbf{H}^1_0)'$. We shall use the standard spaces of divergence-free vector fields
\begin{align*}
\mathcal{D}&:= \{ \mathbf{y} \in C^{\infty}_0(\Omega;\mathbb{R}^2) : \diver \mathbf{y} = 0 \text{ in } \Omega \}, \\   H& := \{ \mathbf{y} \in \mathbf{L}^2 : \diver \mathbf{y} = 0 \text{ in } \Omega \text{ and }  \mathbf{n}\cdot \mathbf{y}=0 \text{ on }  \partial\Omega  \}, \\  
V &:= \{ \mathbf{y} \in \mathbf{H}^1_0 : \diver \mathbf{y} = 0 \text{ in } \Omega \},
\end{align*}
where  $\mathbf{n}$  is the unit outward normal vector on  $\partial \Omega$. The spaces $H$ and $V$ are the closure of the space $\mathcal{D}$ for the $\mathbf{L}^2$- and $\mathbf{H}^1_0$-norms, respectively.  It is known that $V \hookrightarrow  H=H' \hookrightarrow  V'$ where the embeddings are dense and compact. As a consequence,  we recall from,  e.g., \cite{temam1997infinite} that for an open interval $(t_0,t_1)\subset \mathbb{R}_+$,  it holds
$W((t_0,t_1);V,V')  \hookrightarrow  C([t_0,t_1];H)$.  

 Moreover, if we denote the Leray projection on $H$  by $ \Pi :\mathbf{L}^2 \to H$, we have $\Pi (\nabla \mathbf{p})=0$. Due to  \cite[Definition A.1.2, Page 98]{MR2215059},  the Leray projection $\Pi : \mathbf{L}^2 \to H$  can be naturally extended  to a continuous operator from $ \mathbf{H}^{-1}$ to $ V'$. In this case,  there exists  a constant $c_{\Pi}>0$ such that   $\| \Pi  \|_{\mathcal{L}(\mathbf{H}^{-1},V')} \leq c_{\Pi}$. 
 
We also  define the Stokes operator  $\mathcal{A}: D(\mathcal{A}) \to H$ by $\mathcal{A}: = -\Pi \Delta$. We use the same notation for its extension $\mathcal A:V\to V'$, defined by
\[
\langle\mathcal A\mathbf v,\boldsymbol\varphi\rangle_{V',V}
:=(\nabla\mathbf v,\nabla\boldsymbol\varphi)_{L^2},
\qquad \mathbf v,\boldsymbol\varphi\in V.
\]   
The spaces $H$ and $V$ are endowed with the scalar products
\begin{equation*}
(\mathbf{y}, \mathbf{v})_H = (\mathbf{y}, \mathbf{v})_{\mathbf{L}^2}, \quad \text{ and } \quad(\mathbf{y}, \mathbf{v})_V  := \langle \mathcal{A} \mathbf{y}, \mathbf{v}\rangle_{V',V}.
\end{equation*}
To define the weak variational form of the Navier-Stokes equations, we will use the continuous bilinear form
$B : V \times V \to V'$ defined by 
\begin{equation*}
B(\mathbf{y},\mathbf{v}) = \Pi ( \mathbf{y} \cdot \nabla )\mathbf{v}   \quad \text{ with } \quad  \langle B(\mathbf{y},\mathbf{v}), \mathbf{w} \rangle_{V',V} = \sum^2_{i,j =1} \int_{\Omega} y_i \partial_{x_i} v_j w_j dx,
\end{equation*} 
and the trilinear form $b:V\times V\times V \to \mathbb{R}$ defined by 
\begin{equation*}
b(\mathbf{y},\mathbf{v}, \mathbf{w}) := \langle B(\mathbf{y},\mathbf{v}), \mathbf{w} \rangle_{V',V}.
\end{equation*}
It is well-known from e.g., \cite[Lemma 1.3.]{Temam84}, that  for $b$ it holds that
\begin{equation}
\label{e81}
\begin{split}  
b(\mathbf{y},\mathbf{v}, \mathbf{v})  =0   \quad  \text{  and  } \quad  b(\mathbf{y},\mathbf{v}, \mathbf{w}) = -b(\mathbf{y},\mathbf{w}, \mathbf{v}). 
\end{split}
\end{equation} 
Moreover, using standard Sobolev embeddings \cite[Page 293]{Temam84}, we can obtain that 
\begin{equation}
\label{e114}
\begin{aligned}
|b(\mathbf{y},\mathbf{v}, \mathbf{w})|  \leq c \| \mathbf{y}\|^{\frac{1}{2}}_H\| \mathbf{y}\|^{\frac{1}{2}}_V \| \mathbf{v}\|^{\frac{1}{2}}_H\| \mathbf{v}\|^{\frac{1}{2}}_V\|\mathbf{w}\|_V   \qquad  \text{for } \mathbf{y} \in V, \mathbf{v} \in V,  \mathbf{w} \in V, 
\end{aligned}
\end{equation}
where $c$ is a generic constant depending on $\Omega$. We denote  the nonlinear term in the Navier-Stokes equations by 
\begin{equation*}
\mathcal{N}(\mathbf{y}): = B(\mathbf{y},\mathbf{y}). 
\end{equation*}
For a bounded divergence-free reference field $\hat{\mathbf y}$,
we define
\[
\mathcal B(\hat{\mathbf y})\mathbf v
:=B(\hat{\mathbf y},\mathbf v)+B(\mathbf v,\hat{\mathbf y}),
\qquad \mathbf v\in V,
\]
where the two terms are understood in the variational sense. If $\|\hat{\mathbf y}(t)\|_{\mathbf L^\infty}\leq R$ for almost every $t\geq0$, we obtain the bounds
\begin{equation}\label{eq:oseen-bounds}
\|\mathcal B(\hat{\mathbf y}(t))\mathbf v\|_{V'}\leq CR\|\mathbf v\|_V,
\qquad |\langle\mathcal B(\hat{\mathbf y}(t))\mathbf v,\mathbf v\rangle|\leq R\|\mathbf v\|_H\|\mathbf v\|_V.
\end{equation}
Here we use that the transport contribution with advecting field $\hat{\mathbf y}$ vanishes because $\operatorname{div}\hat{\mathbf y}=0$ and $\mathbf v$ has zero trace. Both bounds are uniform in time.

\subsection{Translated system}
In this section, we introduce the translated system obtained by setting  $\mathbf{v}:= \mathbf{y}- \hat{\mathbf{y}}$,  $\mathbf{v}_0 := \mathbf{y}_0- \hat{\mathbf{y}}_0$, $\mathbf q:=\mathbf p-\hat{\mathbf p}$, and subtracting \eqref{e2} from \eqref{e1}. Then, we obtain the following time-varying nonlinear system
\begin{equation}\label{eq:trans}
\begin{aligned}
&\partial_t \mathbf{v}-\nu\Delta \mathbf{v} +( \hat{\mathbf{y}} \cdot  \nabla) \mathbf{v}+  ( \mathbf{v} \cdot  \nabla) \hat{\mathbf{y}}+( \mathbf{v} \cdot  \nabla) \mathbf{v} +\nabla \mathbf{q} = \sum^N_{i =1} u_i  \mathbf{1}_{\mathbf{R}_i}
\quad \text{ in } \mathbb{R}_+ \times \Omega,\\
&\nabla\cdot \mathbf{v}=0,\quad \mathbf{v}=0\ \text{ on }(0,\infty)\times\partial\Omega,\quad
\mathbf{v}(0,\cdot)=\mathbf{v}_0\ \text{in }\Omega.
\end{aligned}
\end{equation}
This system plays a central role in our analysis. In particular, our control objective can now be expressed,  equivalently,  as the local exponential stabilization of \eqref{eq:trans} to zero with respect to the $H$-norm employing the RHC. 

Our stabilizability result is based on concatenating exact controllability controls on a family of finite intervals covering $[ 0, \infty)$. Here we use the exact controllability result given in \cite[Proposition 1.]{MR2103189} and thus, the regularity condition \eqref{RA} is motivated by the one given in \cite[Page 3]{MR2103189}. Therefore, throughout the paper, we impose the following regularity condition for the reference trajectory $\hat{\mathbf{y}}$. 
\begin{assumption} We assume that $(\hat{\mathbf{y}}, \hat{\mathbf{p}})$ is a global smooth solution to \eqref{e2}, in the sense that it holds with constants $\hat{\epsilon}>0$, $\sigma > \frac{6}{5}$, and $R>0$, that 
\begin{equation}
\label{RA}
\tag{RA}
\begin{split}
\|\hat{\mathbf{y}}\|_{L_{\diver}^{\infty}(I_{\infty}(0)\times\Omega ; \mathbb{R}^2)}&+\sup_{\tau \in [0,+\infty)}\| \partial_t \hat{\mathbf{y}}\|_{L^2((\tau,\tau+\hat{\epsilon});\mathbf{L}^{\sigma})}\\&+\sup_{\tau \in [0,+\infty)}\| \nabla \hat{\mathbf{y}}\|_{L^2((\tau,\tau+\hat{\epsilon});L^{3}(\Omega;\mathbb{R}^{4}))} \leq R,
\end{split}
\end{equation}
where 
\begin{equation*}
\begin{split}
&L^{\infty}_{div}(I_{\infty}(0)\times \Omega; \mathbb{R}^2)\\&:=\left\{ \mathbf{y}\in L^{\infty}(I_{\infty}(0)\times \Omega; \mathbb{R}^2) :    \diver \mathbf{y}(t) = 0 \text{ in } \Omega, \text{for a.e. } t \in I_{\infty}(0)  \right\}
\end{split}
\end{equation*}
endowed with the norm $\| \mathbf{y}\|_{L^{\infty}_{div}(I_{\infty}(0)\times \Omega; \mathbb{R}^2)}:=\| \mathbf{y}\|_{L^{\infty}(I_{\infty}(0)\times \Omega; \mathbb{R}^2)}$.
 \end{assumption} 
 Note that due to the fact that $L^{\infty}(I_{\infty}(0)\times \Omega; \mathbb{R}^2)$ is dual of $L^1(I_{\infty}(0)\times \Omega; \mathbb{R}^2)$, we obtain that 
\begin{equation*}
L^{\infty}(I_{\infty}(0)\times \Omega; \mathbb{R}^2) =L_w^{\infty}( I_{\infty}(0); \mathbf{L}^{\infty}) \supset L^{\infty}( I_{\infty}(0); \mathbf{L}^{\infty}),
\end{equation*}
where the subscript $w$ denotes weak-star measurability with respect to the $L^1$-predual, see e.g., \cite[Sections 5.0 and 9.1]{MR1669395}. 

Now we are in the position that we can deal with the well-posedness of \eqref{eq:trans}. 
Let $\hat{\mathbf{y}}$ be the solution  to \eqref{e2} for a  triple $(\hat{\mathbf{y}}_0, \hat{\mathbf{h}},\hat{\mathbf{f}})$. Then for every initial function $\bar{\mathbf{v}}_0$ and forcing  term $\mathbf{f}$,  we consider the auxiliary nonlinear system
\begin{equation}
\label{e16}
\begin{cases}
\partial_t \mathbf{v}(t)+\nu\mathcal{A}\mathbf{v}(t)  + \mathcal{B}(\hat{\mathbf{y}}(t))\mathbf{v}(t)  + \mathcal{N}( \mathbf{v}(t))  =  \mathbf{f}(t)   \qquad   t \in I_{T}(\bar{t}_0),   \\
 \mathbf{v}(\bar{t}_0)= \bar{\mathbf{v}}_0.
\end{cases}
\end{equation}
We introduce the following notion of weak variational solution for \eqref{e16}.

\begin{definition}\label{def:weak}
    Let  $(\bar t_0,T)\in[0,\infty)\times(0,\infty)$ and $(\bar{\mathbf{v}}_0,\mathbf{f})\in H\times L^2(I_{T}(\bar{t}_0);V')$ be given.  Then,  a function $\mathbf{v} \in W(I_{T}(\bar t_0);V,V')$ is referred to as a weak solution of \eqref{e16} if for almost every $t \in I_{T}(\bar{t}_0)$ we have
    \begin{equation}
        \label{e19}
        {\langle\partial_t \mathbf{v}(t),\varphi\rangle}_{V',V}+\nu\,{(\mathcal{A}^{\frac{1}{2}}\mathbf{v}(t),\mathcal{A}^{\frac{1}{2}}\varphi)}_H+{\langle\mathcal{B}(\hat{\mathbf{y}}(t))\mathbf{v}(t)  + \mathcal{N}( \mathbf{v}(t)),\varphi\rangle}_{V',V}= {\langle\mathbf{f}(t),\varphi\rangle}_{V',V}    
        \end{equation}
    for all $\varphi \in V$,  and $\mathbf{v}(\bar{t}_0)= \bar{\mathbf{v}}_0$ is satisfied in $H$.
\end{definition}
For this weak solution, we have the following existence results and energy estimates.

\begin{proposition}\label{Theo2}
Let $\bar t_0\geq0$, $T>0$, $\bar{\mathbf v}_0\in H$, and $\mathbf f\in L^2(I_T(\bar t_0);V')$. Under \eqref{RA}, equation \eqref{e16} admits a unique weak solution $\mathbf v\in W(I_T(\bar t_0);V,V')$. With $E(\bar{\mathbf v}_0,\mathbf f):=\|\bar{\mathbf v}_0\|_H^2+\|\mathbf f\|_{L^2(I_T(\bar t_0);V')}^2$, the following estimates hold:
\begin{align}
\|\mathbf v\|_{C(\overline I_T(\bar t_0);H)}^2+\|\mathbf v\|_{L^2(I_T(\bar t_0);V)}^2&\leq c_1 E(\bar{\mathbf v}_0,\mathbf f),\label{e13}\\
\|\partial_t\mathbf v\|_{L^2(I_T(\bar t_0);V')}^2&\leq c_2\bigl(E(\bar{\mathbf v}_0,\mathbf f)+E(\bar{\mathbf v}_0,\mathbf f)^2\bigr),\label{eq:time-derivative-bound}\\
\|\mathbf v(\bar t_0+T)\|_H^2&\leq c_3 \left(\|\mathbf v\|_{L^2(I_T(\bar t_0);H)}^2+\|\mathbf f\|_{L^2(I_T(\bar t_0);V')}^2\right).\label{e14}
\end{align}
The constants depend on $T,\nu,R,\Omega$, but not on $\bar t_0$, $\bar{\mathbf v}_0$, or $\mathbf f$.
\end{proposition}
\begin{proof}
By testing \eqref{e19} with $\mathbf v$ and using \eqref{e81} and \eqref{eq:oseen-bounds}, we obtain
\[
\frac{d}{dt}\|\mathbf v\|_H^2+\nu\|\mathbf v\|_V^2
\leq C\|\mathbf v\|_H^2+C\|\mathbf f\|_{V'}^2.
\]
Using Gronwall's inequality, we obtain \eqref{e13}, where $c_1$ may depend on $T$ but is independent of $\bar t_0$, $\bar{\mathbf v}_0$, and $\mathbf f$. We recall the standard existence and uniqueness result in two dimensions; see \cite{MR1318914,BoFab13}. Using \eqref{e114}, we obtain
\begin{equation}\tag{11a}\label{eq:nonlinear-energy-bound}
\|\mathcal N(\mathbf v)\|_{L^2(V')}^2
\leq C\|\mathbf v\|_{L^\infty(H)}^2\|\mathbf v\|_{L^2(V)}^2\leq C E(\bar{\mathbf v}_0,\mathbf f)^2,
\end{equation}
where the time interval is $I_T(\bar t_0)$. We use \eqref{e16}, \eqref{eq:oseen-bounds}, and \eqref{eq:nonlinear-energy-bound} to obtain \eqref{eq:time-derivative-bound}.

Finally, we test \eqref{e19} with $(t-\bar t_0)\mathbf v$ and again use \eqref{e81}, \eqref{eq:oseen-bounds}, and Young's inequality to obtain
\[
\frac{d}{dt}\bigl((t-\bar t_0)\|\mathbf v(t)\|_H^2\bigr)
\leq \bigl(1+C(t-\bar t_0)\bigr)\|\mathbf v(t)\|_H^2
+C(t-\bar t_0)\|\mathbf f(t)\|_{V'}^2.
\]
We integrate over $I_T(\bar t_0)$ and divide by $T$ to obtain
\[
\|\mathbf v(\bar t_0+T)\|_H^2
\leq (T^{-1}+C)\|\mathbf v\|_{L^2(I_T(\bar t_0);H)}^2
+C\|\mathbf f\|_{L^2(I_T(\bar t_0);V')}^2.
\]
This proves \eqref{e14}, with a constant that may depend on $T^{-1}$.\qedhere
\end{proof}

\section{Stabilizability of RHC}\label{sec:stabilization}
In this section, we review essential preliminary results on the stabilizability of \eqref{eq:trans}, following the framework presented in \cite{Azmi5}. We first describe the design and placement of actuators, which are modeled as indicator functions. For simplicity, we assume that $\omega$ is a nonempty open rectangle of the form
\begin{equation}\label{e99}
\omega:=\prod_{i=1}^2(a_i,b_i)\subset\Omega.
\end{equation}
We next consider a uniform partition of $\omega$ into sub-rectangles. For this purpose, we introduce $d_i+1$ uniformly distributed grid points in $[a_i,b_i]$ by defining $\xi_{i,k_i}=a_i+k_i\frac{\bar I_i}{d_i}$, where $\bar I_i:=b_i-a_i$ and $k_i=0,\ldots,d_i$ for $d_i\in\mathbb N$. For each $i\in\{1,2\}$, the interval $(a_i,b_i)$ is now divided into 
the subintervals
\[
I_{i,k_i} :=\left(\xi_{i,k_i},\xi_{i,k_i+1} \right), \qquad k_i=0,\dots,d_i-1.
\]
This yields a total of \(M_a:=\prod_{i=1}^2 d_i\) sub-rectangles, namely
\[
\{R_j:j=1,\dots,M_a\}:=\{I_{1,k_1}\times I_{2,k_2}: k_i=0,\dots,d_i-1,\ i=1,2\}.
\]
For each sub-rectangle $R_j$, $j=1,\dots,M_a$, we define the scalar-valued function $\phi_j\in L^2(\Omega;\mathbb R)$ by
$\phi_j:=\|1_{R_j}\|_{L^2(\Omega;\mathbb R)}^{-1}1_{R_j}$, where \(1_{R_j}\) denotes the indicator function of $R_j$. We use the componentwise actuator profiles
\begin{equation}\label{e128}
\mathbf{1}_{\mathbf R_{2j-1}}:=(\phi_j,0),\qquad
\mathbf{1}_{\mathbf R_{2j}}:=(0,\phi_j),\qquad j=1,\dots,M_a.
\end{equation}
Thus $N:=2M_a$ and $\mathcal U_\omega:=\{\mathbf{1}_{\mathbf R_i}:i=1,\dots,N\}$. Let $P_{M_a}$ be the scalar $L^2$-projection onto $\operatorname{span}\{\phi_j\}$ and $P_{\mathcal U}$ the vector projection onto $\operatorname{span}\mathcal U_\omega$. Then $P_{\mathcal U}(w_1,w_2)=(P_{M_a}w_1,P_{M_a}w_2)$.

Let $\Upsilon>0$ be the linear stabilization threshold of \cite[Proposition 4]{Azmi5}, chosen as in Appendix~\ref{proof_theo1}, with the two-dimensional spaces and reference regularity \eqref{RA}. This stabilization result assumes the following actuator approximation condition:
\begin{equation*}\tag{COAC}\label{eq:actuator-approximation}
\|\Pi1_\omega(I-P_{\mathcal U})1_\omega\|_{\mathcal L(\mathbf L^2,V')}^2<\Upsilon^{-1}.
\end{equation*}
To verify this condition, we observe that on each cell $R_j$, $P_{\mathcal U}\boldsymbol\psi$ equals the cell average $\boldsymbol\psi_{R_j}:=|R_j|^{-1}\int_{R_j}\boldsymbol\psi\,dx$. Setting $h_{\max}:=\max_{i=1,2}(\bar I_i/d_i)$ and applying the Poincar\'e inequality to $\boldsymbol\psi-\boldsymbol\psi_{R_j}$, which has zero mean on $R_j$, we obtain, after summing over the cells,
\[
\|(I-P_{\mathcal U})\boldsymbol\psi\|_{\mathbf L^2(\omega)}
\leq \frac{h_{\max}}{\pi}\|\nabla\boldsymbol\psi\|_{L^2(\omega)},
\qquad \boldsymbol\psi\in\mathbf H^1_0.
\]
By the self-adjointness of $P_{\mathcal U}$ and the Cauchy--Schwarz inequality, for $\mathbf v\in\mathbf L^2$, we obtain
\[
\bigl|\langle1_\omega(I-P_{\mathcal U})1_\omega\mathbf v,\boldsymbol\psi\rangle\bigr|
=\bigl|(\mathbf v,(I-P_{\mathcal U})\boldsymbol\psi)_{\mathbf L^2(\omega)}\bigr|
\leq \frac{h_{\max}}{\pi}\|\mathbf v\|_{\mathbf L^2}\|\nabla\boldsymbol\psi\|_{L^2}.
\]
Taking the supremum over $\boldsymbol\psi$ in the $\mathbf H^1_0$ unit ball and using the bound for $\Pi:\mathbf H^{-1}\to V'$ yields
\[
\|\Pi1_\omega(I-P_{\mathcal U})1_\omega\|_{\mathcal L(\mathbf L^2,V')}^2
\leq \frac{c_\Pi^2}{\pi^2}\max_{i=1,2}\frac{\bar I_i^2}{d_i^2}.
\]

For the theorem below we take $d_1=d_2=d$, so $M_a=d^2$ and $N=2d^2$. The bound above yields the sufficient actuator count stated there. For partitions with $d_1\ne d_2$, the projection error is controlled by the largest cell side length $h_{\max}$. Hence the same estimate ensures \eqref{eq:actuator-approximation} whenever each cell side is smaller than $\pi/(c_\Pi\sqrt{\Upsilon})$. All constants may depend on the fixed family and on $\omega,\Omega$.

\begin{theorem}
\label{theo1}
Suppose that $(\hat{\mathbf y},\hat{\mathbf p})$ is a global reference solution of \eqref{e2} satisfying \eqref{RA}, and let $\lambda>0$. There exists a threshold $\Upsilon=\Upsilon(\lambda,\nu,\hat{\mathbf y},\omega)>0$ such that, for the partitions with $d_1=d_2=d$ specified above, whenever
\begin{equation}\label{e34}
N = 2M_a > 2\pi^{-2}\bar I^{2}c_\Pi^{2}\Upsilon
\qquad \text{with } \quad 
\bar I := \max_{1\le i\le 2}\bar I_i,
\end{equation}
there exist a family of continuous operators \(\tilde{\mathbf K}_\lambda(t):H\to \spn{\mathcal U_\omega}\), uniformly bounded by a constant $c_{\tilde{\mathbf K}}=c_{\tilde{\mathbf K}}(\hat{\mathbf y}, \omega,\mathcal U_\omega ,\lambda,\nu)$, with $\mathcal U_\omega:=\{\mathbf 1_{\mathbf R_i}:i=1,\dots,N\}$, and a radius \(r_s>0\) such that, for every pair \((\bar t_0,\bar{\mathbf v}_0)\in\mathbb R_+\times \mathrm B_{r_s}(0)\), where \(\mathrm B_{r_s}(0)\subset H\), the nonlinear system
\begin{equation}
\label{e41a}
\begin{cases}
\partial_t \mathbf v(t)+\nu \mathcal A \mathbf v(t)+\mathcal B(\hat{\mathbf y}(t))\mathbf v(t)+\mathcal N(\mathbf v(t))
=\Pi \tilde{\mathbf K}_\lambda(t)\mathbf v(t), & t\in I_\infty(\bar t_0),\\
\mathbf v(\bar t_0)=\bar{\mathbf v}_0,
\end{cases}
\end{equation}
is well-posed, and its solution satisfies
\begin{equation}
\label{e105}
\|\mathbf v(t)\|_H^2 \le \Theta_1 e^{-\lambda(t-\bar t_0)}\|\bar{\mathbf v}_0\|_H^2
\qquad \text{for all } t\ge \bar t_0,
\end{equation}
where $\Theta_1=\Theta_1(\hat{\mathbf y},\omega,\mathcal U_\omega,\lambda,\nu)$.

\end{theorem}

\begin{proof}
The proof is provided in Appendix \ref{proof_theo1}. \end{proof}

From both numerical and theoretical perspectives, it is advantageous to compute the RHC using the translated system \eqref{eq:trans}. To this end, we define the following performance index function for all
 $T \in (0,\infty]$ and $(\bar{t}_0,\bar{\mathbf{v}}_0) \in  \mathbb{R}_+ \times H$
 \begin{equation}
\label{PIF_trans_sys}
J^{\mathrm{tr}}_{T}(\mathbf{u}; \bar{t}_0,  \bar{\mathbf{v}}_0):= \frac{1}{2}\int_{\bar{t}_0}^{ \bar{t}_0+T}\|\mathbf{v}(t)\|^2_{L^2(\Omega; \mathbb{R}^2)}\,dt+ \frac{\beta}{2}\int^{\bar{t}_0+T}_{\bar{t}_0} |\mathbf{u}(t)|^2_{2}dt,
\end{equation} 
subject to 
\begin{equation}
\label{e16n}
\begin{cases}
\partial_t \mathbf{v}(t)+\nu\mathcal{A}\mathbf{v}(t)  + \mathcal{B}(\hat{\mathbf{y}}(t))\mathbf{v}(t)  + \mathcal{N}( \mathbf{v}(t))  =  \mathbf{B} \mathbf{u}(t)   \qquad   t \in I_{T}(\bar{t}_0),   \\
 \mathbf{v}(\bar{t}_0)= \bar{\mathbf{v}}_0,
\end{cases}
\end{equation}
where $\mathbf{B}:= [\Pi \mathbf{1}_{\mathbf{R}_1}, \dots, \Pi \mathbf{1}_{\mathbf{R}_N}]$. That is, $\mathbf{B}\mathbf{u}=\sum_{i=1}^N u_i\,\Pi \mathbf{1}_{\mathbf{R}_i}$.
The infinite-horizon problem \eqref{opinf} can then be reformulated as
\begin{equation}
\label{Op-inf-trans}
\tag*{\ensuremath{\mathrm{OP}^{\mathrm{tr}}_{\infty}(\mathbf{v}_0)}}
\inf_{\mathbf{u} \in L^{2}(I_{\infty}(0); \mathbb{R}^N) }\{ J^{\mathrm{tr}}_{\infty}(\mathbf{u};0,\mathbf{v}_0) \text{ s.t. } \eqref{e16n} \text{ and } \mathbf{v}_0:=\mathbf{y}_0- \hat{\mathbf{y}}_0\}.
\end{equation}
Correspondingly, the finite-horizon problems are defined as
\begin{equation}
\label{OpT_tran}
\tag*{\ensuremath{\mathrm{OP}^{\mathrm{tr}}_{T}(\bar{t}_0, \bar{\mathbf{v}}_0)}}
\min_{\mathbf{u} \in L^{2}(I_{T}(\bar{t}_0); \mathbb{R}^N) }\{ J^{\mathrm{tr}}_{T}(\mathbf{u};\bar{t}_0, \bar{\mathbf{v}}_0) \text{ s.t. } \eqref{e16n}\}.
\end{equation}
With this formulation, Algorithm \ref{RHA} can equivalently be restated as Algorithm \ref{RHA2}.
\begin{algorithm}[htbp]
\caption{RHC($\delta,T$)}\label{RHA2}
\begin{algorithmic}[1]
\REQUIRE{Final time $T_{\infty} \in \mathbb{R}_{\geq 0} \cup \{ \infty\}$, sampling time $\delta$, the prediction horizon $T\geq \delta$, 
a reference trajectory $\hat{\mathbf{y}}$ and the initial state $\mathbf{v}_0 \coloneqq \mathbf{y}_0-\hat{\mathbf{y}}_0$.}
\ENSURE{The stabilizing RHC~$\mathbf{u}_{rh}$, nondecreasing sequence~$\{ t_i \}_{i\in\mathbb{N}}$.}
\STATE Set $\bar t_0 \gets 0$, $\bar\bv_0 \gets \bv_0$, and $i\gets 0$.
\WHILE{$ {\bar{t}_{0}} <T_\infty$}
\STATE Find the optimal solution $(\mathbf{v}_T^*(\cdot;\bar{t}_0, \bar{\mathbf{v}}_0), \mathbf{u}^*_T(\cdot;\bar{t}_0, \bar{\mathbf{v}}_0))$
over the time horizon $I_T(\bar t_0)$ by solving the open-loop problem~\eqref{OpT_tran}.
\STATE Set $t_{i+1}\gets \min\{\bar t_0+\delta,T_\infty\}$.
\STATE For all $\tau\in[\bar{t}_0,t_{i+1})$ set $\mathbf{v}_{rh}(\tau):= \mathbf{v}_T^*(\tau;\bar{t}_0,\bar{\mathbf{v}}_0)$ and $\mathbf{u}_{rh}(\tau):=\mathbf{u}^*_T(\tau;\bar{t}_0,\bar{\mathbf{v}}_0)$.
 \STATE Update $\bar\bv_0 \gets \bv^{*}_T(t_{i+1};\bar t_0,\bar\bv_0)$, $\bar t_0\gets t_{i+1}$, and $i\gets i+1$.
\ENDWHILE
\end{algorithmic}
\end{algorithm}
Next, we investigate the finite-horizon problem \eqref{OpT_tran}.

\begin{proposition}
\label{prop1}
 For every $(T,\bar t_0,\bar{\mathbf v}_0)\in(0,\infty)\times[0,\infty)\times H$, \eqref{OpT_tran} admits a solution $\mathbf{u}^* \in L^2(I_T(\bar{t}_0);\mathbb{R}^N)$. Moreover, this solution satisfies the first-order optimality condition 
\begin{equation}
 \beta \mathbf{u}^{*}(t)-\mathbf{B}^* \mathbf{w}^*(t) = 0 \qquad   t \in I_{T}(\bar{t}_0),
\end{equation}
where $\mathbf w^*\in L^2(I_T(\bar t_0);V)\cap C(\overline I_T(\bar t_0);H)$, with $\partial_t\mathbf w^*\in L^{4/3}(I_T(\bar t_0);V')$, is the solution to the adjoint system
\begin{equation}\label{eq:adjoint}
\begin{cases}
-\partial_t \mathbf{w}^*(t)+\nu\mathcal{A}\mathbf{w}^*(t)  + \mathcal{B}^*(\hat{\mathbf{y}}(t)+\mathbf{v}^*(t))\mathbf{w}^*(t) = - \mathbf{v}^*(t)   \qquad   t \in I_{T}(\bar{t}_0),   \\
 \mathbf{w}^*(\bar{t}_0+T)= 0,
\end{cases}
\end{equation}
where $\mathcal{B}^*$ denotes the adjoint operator of  $\mathcal{B}$, and is defined as  
\begin{equation*}
   \mathcal{B}^*(\mathbf{y})\mathbf{v} \coloneqq \Pi \left( (\nabla \mathbf{y})^t \mathbf{v} - ( \mathbf{y} \cdot \nabla )\mathbf{v} \right). 
\end{equation*}
Additionally, $\mathbf{v}^* \in W(I_{T}(\bar{t}_0);V,V')$ denotes the state trajectory associated with the optimal control $\mathbf{u}^*$, that is, the solution of \eqref{e16n} corresponding to $\mathbf{u}^*$.
\end{proposition}
\begin{proof}
We use the direct method of the calculus of variations. The control cost bounds a minimizing sequence in $L^2(I_T(\bar t_0);\mathbb R^N)$, and estimates \eqref{e13}--\eqref{eq:time-derivative-bound} of Proposition~\ref{Theo2}, with $\mathbf f=\mathbf B\mathbf u$, bound the corresponding states in $W(I_T(\bar t_0);V,V')$. We extract a subsequence with weakly convergent controls and strongly convergent states in $L^2(I_T(\bar t_0);H)$, pass to the limit in the state equation, and use weak lower semicontinuity of the objective to obtain a minimizer; see also \cite[Section 2]{HinKun01}. We derive the first-order optimality conditions as in \cite[Section 3]{HinKun01}.\end{proof}

In the next theorem, we present the main results of this section, namely the exponential stabilizability and suboptimality of the RHC generated by Algorithm \ref{RHA2}. The latter is characterized in terms of the finite- and infinite-horizon value functions, defined as follows:

\begin{definition}
For any $\mathbf{v}_0 \in  H$  the infinite-horizon value function $V_\infty:H\to[0,\infty]$ is defined by
\begin{equation*}
V_{\infty}(\mathbf{v}_0):= \inf_{ \mathbf{u} \in L^2(I_{\infty}(0); \mathbb{R}^N)}\{J^{\mathrm{tr}}_{\infty}(\mathbf{u};0,\mathbf{v}_0) \text{  s.t. \eqref{e16n}} \}.
\end{equation*}
Similarly, for every $(T,\bar t_0,\bar{\mathbf v}_0)\in(0,\infty)\times[0,\infty)\times H$,  the finite-horizon value function $V_{T}: \mathbb{R}_+ \times H \to \mathbb{R}_+$ is defined by
\begin{equation*}
V_{T}(\bar{t}_0, \bar{\mathbf{v}}_0):= \min_{ \mathbf{u} \in L^2(I_{T}(\bar{t}_0); \mathbb{R}^N)}\{J^{\mathrm{tr}}_{T}(\mathbf{u}; \bar{t}_0,\bar{\mathbf{v}}_0 ) \text{ s.t. \eqref{e16n}} \}.
\end{equation*}
\end{definition}
We are now in a position to state the main results.

\begin{theorem}
\label{theo3}

Suppose that $(\hat{\mathbf y},\hat{\mathbf p})$ is a global reference solution of \eqref{e2} satisfying \eqref{RA}, with compatible forcing, boundary data, and initial state. Assume the actuator hypotheses of Theorem~\ref{theo1}. Then, for any given $\delta>0$, and  the fixed set of actuators $\mathcal{U}_{\omega} =\mathcal{U}_{\omega}(N)$,  there exist numbers $T^* =T^*(\delta,\mathcal{U}_{\omega})>\delta$ and $\alpha =  \alpha(\delta,\mathcal{U}_{\omega})\in(0, 1)$ such that:  For every fixed  prediction horizon $T \geq T^*$,   and every  $\mathbf{v}_0:=\mathbf{y}_0-\hat{\mathbf{y}}_0 \in   \mathrm{B}_{d}(0)$ with $d = d(T)>0$, the RHC $\mathbf{u}_{rh}$ obtained by Algorithm \ref{RHA2} satisfies the suboptimality inequalities  
\begin{equation}
\label{e110}
\alpha  V_{\infty}(\mathbf{v}_0) \leq \alpha J^{\mathrm{tr}}_{\infty}(\mathbf{u}_{rh};0,\mathbf{v}_0) \leq V_{T}(0,\mathbf{v}_0)\leq   V_{\infty}(\mathbf{v}_0)
\end{equation}
and  the exponentially stable estimate
\begin{equation}
\label{e111}
\| \mathbf{v}_{rh}(t)\|^2_{H} \leq c_He^{-\zeta t}\|\mathbf{v}_0\|^2_{H} \quad  \text{ for }  t\geq 0,
\end{equation}
where  $\zeta$ and  $c_H$  depend on $\mathcal{U}_{\omega}$, $\delta$, and $T$, but are independent of $\mathbf{v}_0$.
\end{theorem}

\begin{proof}
We verify the $H$-norm counterparts of properties P1--P4 in \cite[Section 3 and Theorem 1]{Azmi5}. We write
\[
\ell(\mathbf v,\mathbf u)=\tfrac12\|\mathbf v\|_H^2+\tfrac\beta2|\mathbf u|_2^2,
\qquad \alpha_\ell=\tfrac12\min\{1,\beta\}>0.
\]
To verify P1, let $(\mathbf v,\bar{\mathbf u})$ be the exponentially stabilizing state--control pair provided by Theorem~\ref{theo1}, where $\bar{\mathbf u}=(\bar u_1,\ldots,\bar u_N)$ is defined by
\[
\bar u_i(t):=\bigl(\tilde{\mathbf K}_\lambda(t)\mathbf v(t),\mathbf1_{\mathbf R_i}\bigr)_{\mathbf L^2},
\qquad i=1,\ldots,N.
\]
Since the actuator profiles are orthonormal, $\mathbf B\bar{\mathbf u}(t)=\Pi\tilde{\mathbf K}_\lambda(t)\mathbf v(t)$ and the uniform feedback bound gives $|\bar{\mathbf u}(t)|_2=\|\tilde{\mathbf K}_\lambda(t)\mathbf v(t)\|_{\mathbf L^2}\leq C_u\|\mathbf v(t)\|_H$, with $C_u>0$ independent of the initial time. Using $\bar{\mathbf u}$ as an admissible comparison control in \eqref{PIF_trans_sys}, we obtain $V_T(\bar t_0,\bar{\mathbf v}_0)\leq J_T^{\mathrm{tr}}(\bar{\mathbf u};\bar t_0,\bar{\mathbf v}_0)$. Combining the control bound with \eqref{e105} and integrating yields
\begin{equation}\label{e7}
V_T(\bar t_0,\bar{\mathbf v}_0)\leq\gamma(T)\|\bar{\mathbf v}_0\|_H^2,\qquad
\gamma(T)=\frac{\Theta_1(1+\beta C_u^2)}{2\lambda}(1-e^{-\lambda T}),
\end{equation}
for $\|\bar{\mathbf v}_0\|_H<r_s$, uniformly in $\bar t_0$. This function is continuous, nondecreasing, and bounded. Thus \eqref{e7} verifies P1.

For P2, we apply Proposition~\ref{Theo2} with $\mathbf f=\mathbf B\mathbf u$. Estimate \eqref{e13} gives finite-horizon well-posedness and the required uniform $H$-norm bound, while \eqref{e13}--\eqref{eq:time-derivative-bound} give uniform bounds in $W(I_T(\bar t_0);V,V')$ on bounded sets of initial states and controls. P3 follows from Proposition~\ref{prop1}.

Finally, we use \eqref{e14} on an interval of length $\delta$ to obtain P4:
\[
\|\mathbf v(\bar t_0+\delta)\|_H^2
\leq C_\delta\int_{\bar t_0}^{\bar t_0+\delta}
\bigl(\|\mathbf v(t)\|_H^2+|\mathbf u(t)|_2^2\bigr)\,dt.
\]
All constants are independent of the initial time. Since properties P1--P4 have been verified, we apply the argument of \cite[Theorem 1]{Azmi5} with the $H$-norm in place of the $V$-norm. It follows that there exists a radius $d(T)>0$ such that, for every $\mathbf v_0\in\mathrm B_{d(T)}(0)\subset H$, \eqref{e110} and the following decay estimate hold:
\begin{equation}\label{e112}
V_T(t_i,\mathbf v_{rh}(t_i))\leq e^{-\zeta t_i}V_T(0,\mathbf v_0).
\end{equation}
To prove \eqref{e111}, we first use P4, \eqref{e7}, and \eqref{e112} to obtain
\[
\|\mathbf v_{rh}(t_{i+1})\|_H^2
\leq C V_T(t_i,\mathbf v_{rh}(t_i))
\leq C\gamma(T)e^{-\zeta t_i}\|\mathbf v_0\|_H^2,
\qquad i\geq0.
\]
Here the cost on $[t_i,t_{i+1}]$ is bounded by the finite-horizon optimal cost. Since $t_{i+1}=t_i+\delta$, we obtain $\|\mathbf v_{rh}(t_i)\|_H^2\leq c_s e^{-\zeta t_i}\|\mathbf v_0\|_H^2$ for all $i\geq0$, with a constant $c_s\geq1$ that also covers the initial state. Now, for any $t\geq0$, we choose $i$ such that $t_i\leq t<t_{i+1}$. We apply \eqref{e13} on this interval with $\mathbf f=\mathbf B\mathbf u_{rh}$ and use the coercivity of the stage cost and \eqref{e7} to obtain
\[
\begin{aligned}
\|\mathbf v_{rh}(t)\|_H^2
&\leq C_\delta\left(\|\mathbf v_{rh}(t_i)\|_H^2+
\|\mathbf u_{rh}\|_{L^2((t_i,t_{i+1});\mathbb R^N)}^2\right)\\
&\leq C_\delta\left(1+\frac{\gamma(T)}{\alpha_\ell}\right)
\|\mathbf v_{rh}(t_i)\|_H^2
\leq c_H e^{-\zeta t}\|\mathbf v_0\|_H^2.
\end{aligned}
\]
In the last step, we use $e^{-\zeta t_i}\leq e^{\zeta\delta}e^{-\zeta t}$ and take $c_H=C_\delta(1+\gamma(T)/\alpha_\ell)c_s e^{\zeta\delta}$. This proves \eqref{e111} for every $t\geq0$.\end{proof}

\begin{remark}
Since the performance index functions \eqref{e49} and \eqref{PIF_trans_sys} agree, the finite- and infinite-horizon problems \eqref{opT}, \eqref{OpT_tran} and \eqref{opinf}, \eqref{Op-inf-trans} are equivalent when initialized with $\bar{\mathbf{v}}_0 := \bar{\mathbf{y}}_0-\hat{\mathbf{y}}(\bar{t}_0)$ and $\mathbf{v}_0 := \mathbf{y}_0-\hat{\mathbf{y}}_0$, and admit the same optimal controls. Hence, Algorithms~\ref{RHA} and \ref{RHA2} generate the same RHC controls. In particular, Theorem~\ref{theo3} remains valid for Algorithm~\ref{RHA}; that is, \eqref{e110} holds with $J^o_{\infty}$ in place of $J^{\mathrm{tr}}_{\infty}$, and \eqref{main-stability-result} holds as well.
\end{remark}

\section{Full-order discretization and optimization}\label{sec:disc_and_impl}
In this section, we describe the implementation details of Algorithm \ref{RHA2}, including the discretization and optimization procedures.

For discretization we use a mixed velocity--pressure formulation with unprojected actuator forcing $\mathbf f_c=\sum_i u_i\mathbf1_{\mathbf R_i}$. The velocity belongs to $L^2(I_T(\bar t_0);\mathbf H^1_0)\cap C(\overline I_T(\bar t_0);H)$. The perturbation pressure $\mathbf q$ from \eqref{eq:trans} is recovered as a distribution in time with values in $L^2_0(\Omega)$. We extend the trilinear form $b$ from Section~\ref{sec:prelim} to $H^1(\Omega;\mathbb R^2)^3$ by the same integral formula. For $\boldsymbol\varphi\in\mathbf H^1_0$ and $\xi\in L^2_0(\Omega)$ the equations read, in time distributions,
\begin{equation}\label{e20}
\begin{aligned}
(\partial_t  \mathbf v,\boldsymbol\varphi)_{\mathbf L^2}
&+\nu(\nabla\mathbf v,\nabla\boldsymbol\varphi)_{L^2}
+b(\hat{\mathbf y},\mathbf v,\boldsymbol\varphi)
+b(\mathbf v,\hat{\mathbf y},\boldsymbol\varphi)\\
&+b(\mathbf v,\mathbf v,\boldsymbol\varphi)
-(\mathbf q,\operatorname{div}\boldsymbol\varphi)_{L^2}
=(\mathbf f_c,\boldsymbol\varphi)_{\mathbf L^2},\\
(\operatorname{div}\mathbf v,\xi)_{L^2}&=0,
\end{aligned}
\end{equation}
with $\mathbf v(\bar t_0)=\bar{\mathbf v}_0$.

We derive the corresponding discrete formulation based on this velocity–pressure formulation \eqref{e20}. Let ${\cal T}_h$ be a quasi-uniform family of triangulations of  $\Omega$ into triangles, in the sense of~\cite{brenner2008}. For simplicity, we assume that $\Omega$ is polygonal, such that it can be exactly represented by the triangulation. For spatial discretization, we use the lowest-order Taylor–Hood finite element (FE) pair~\cite{HoodTaylor1974}
\begin{align*}
V_h &:=\{ \phi_h\in C(\bar{\Omega};\mathbb{R}^2): \phi_h|_K \in (P_2)^2 \,\quad \forall K\in {\cal T}_h, \, \phi_h=0 \text{ on } \partial\Omega \} \subset \mathbf{H}^1_0,\\
Q_h &:=\{ \phi_h\in C(\bar{\Omega};\mathbb{R})\cap L^2_0(\Omega;\mathbb{R}): \phi_h|_K \in P_1\, \quad \forall K\in {\cal T}_h \}.
\end{align*}
For temporal discretization on each prediction horizon $I_T(\bar t_0)$, we use the local time grid $t_n:=\bar t_0+n\Delta t$, $n=0,\ldots,N_T$, with $\Delta t=T/N_T$. These local time nodes are distinguished from the RHC sampling times $t_i$.  We then apply a semi-implicit time-stepping scheme, where the nonlinear and linear convection terms are treated explicitly. That is, given a discretely divergence-free initial approximation $\mathbf v_h^0$ of $\bar{\mathbf v}_0$, we successively compute for $n=1,...,N_T$ the solution $(\mathbf{v}_h^n, \mathbf q_h^n) \in V_h \times Q_h$ to the following systems.
\begin{equation}
\begin{aligned}\label{eq:disc}
{\Delta t}^{-1} (\bv_h^{n} - \bv_h^{n-1}, \varphi_h)_{\mathbf L^2} + \nu (\nabla \bv_h^{n}, \nabla \varphi_h)_{L^2} - (\mathbf q_h^n, \div \varphi_h)_{L^2(\Omega;\mathbb{R})} &= \\ (\mathbf f_c(t_n), \varphi_h)_{\mathbf L^2} -((\bv_h^{n-1}+\hat{\by}_h^{n-1})\cdot \nabla \bv_h^{n-1}, \varphi_h)_{\mathbf L^2} -(\bv_h^{n-1} \cdot \nabla \hat{\by}_h^{n-1}&, \varphi_h)_{\mathbf L^2},  &&\forall \varphi_h \in V_h\\
(\div \bv_h^n, \xi_h)_{L^2(\Omega;\mathbb{R})} &= 0  \, &&\forall \xi_h \in Q_h.
\end{aligned}
\end{equation}
In this formulation, the matrices on the left-hand side of~\eqref{eq:disc} are time-independent and thus, need to be computed once only. Specifically, given the nodal bases $\{\varphi_k: k=1,\ldots,n_v\}$ and $\{\xi_k: k=1,\ldots,n_p\}$ of the finite-dimensional spaces $V_h$ and  $Q_h$, respectively, we define 
\begin{equation}\label{eq:fe_mat}
    \begin{aligned}
        A &:= (a_{ij})_{i,j=1}^{n_v}, \quad   \,a_{ij} :=(\nabla \varphi_j,\nabla \varphi_i)_{L^2}, \quad
        M:= (m_{ij})_{i,j=1}^{n_v},\\  m_{ij}&:=(\varphi_j,\varphi_i)_{\mathbf L^2}, \quad  D:= (b_{ij})_{i,j=1}^{n_v,n_p}, \, \quad  b_{ij}:=(\xi_j,\div \varphi_i)_{L^2(\Omega;\mathbb{R})}.
    \end{aligned}
\end{equation}
The convection tensor $C$ is assembled from $c_{ijk}:=b(\varphi_j,\varphi_k,\varphi_i)$, with
\[
[C(\mathbf a\otimes\mathbf b)]_i=\sum_{j,k=1}^{n_v}c_{ijk}a_jb_k,
\qquad C\in\mathbb R^{n_v\times n_v^2}.
\]
The load matrix $F\in\mathbb R^{n_v\times N}$ has entries $F_{ij}=(\mathbf1_{\mathbf R_j},\varphi_i)_{\mathbf L^2}$.

As the reference trajectory $\hat{\mathbf y}$ may have nonhomogeneous boundary values, its Oseen matrix is assembled from the full reference FE function:
\[
(O_n)_{ij}=b(\hat{\mathbf y}_h(t_n),\varphi_j,\varphi_i)
+b(\varphi_j,\hat{\mathbf y}_h(t_n),\varphi_i),\qquad
G_n(\mathbf x)=C(\mathbf x\otimes\mathbf x)+O_n\mathbf x.
\]
Here $O_n=O(t_n)$, where $O(t)$ is assembled by the same formula with $t_n$ replaced by $t$.
The state scheme \eqref{eq:disc} becomes
\begin{equation}\label{eq:v_disc}
\begin{aligned}
(\Delta t^{-1}M+\nu A)\mathbf v^n-D\mathbf q^n
&=\Delta t^{-1}M\mathbf v^{n-1}-G_{n-1}(\mathbf v^{n-1})+F\mathbf u^n,\\
D^\top\mathbf v^n&=0,\qquad n=1,\ldots,N_T.
\end{aligned}
\end{equation}
Here $\mathbf v^0$ is a discretely divergence-free approximation of the initial perturbation. The vector $\mathbf q^n\in\mathbb R^{n_p}$ contains the perturbation-pressure coefficients. The discrete adjoint is derived from the transpose of the state Jacobian:
\[
DG_n(\mathbf x)[\delta\mathbf x]=C(\delta\mathbf x\otimes\mathbf x)+C(\mathbf x\otimes\delta\mathbf x)+O_n\delta\mathbf x.
\]
We consider the following time-discrete objective functional
\begin{equation}\label{eq:discrete-cost}
J_h(\mathbf u)=\frac{\Delta t}{2}\sum_{n=1}^{N_T}(\mathbf v^n)^\top M\mathbf v^n
+\frac{\beta\Delta t}{2}\sum_{n=1}^{N_T}|\mathbf u^n|_2^2.
\end{equation}
The initial state is fixed during each optimization. Differentiating the Lagrangian obtained by adding $\Delta t\sum_n(\mathbf w^n)^\top$ times the state residual, with adjoint pressure multipliers $\mathbf r^n\in\mathbb R^{n_p}$ for the divergence constraint, gives
\begin{equation}\label{eq:w_disc}
\begin{aligned}
(\Delta t^{-1}M+\nu A)^\top\mathbf w^{N_T}-D\mathbf r^{N_T}&=-M\mathbf v^{N_T},\\
(\Delta t^{-1}M+\nu A)^\top\mathbf w^n-D\mathbf r^n
&=\bigl(\Delta t^{-1}M-DG_n(\mathbf v^n)\bigr)^\top\mathbf w^{n+1}-M\mathbf v^n,
\quad n=N_T-1,\ldots,1,\\
D^\top\mathbf w^n&=0,\quad n=1,\ldots,N_T.
\end{aligned}
\end{equation}

We denote the discrete control space by $\mathcal U_{\Delta t}:=(\mathbb R^N)^{N_T}$ and write $\mathbf u=(\mathbf u^1,\ldots,\mathbf u^{N_T})$, where $\mathbf u^n\in\mathbb R^N$. For $\mathbf u,\widetilde{\mathbf u}\in\mathcal U_{\Delta t}$, we define the inner product and induced norm
\begin{equation}\label{eq:discrete-control-space}
\langle\mathbf u,\widetilde{\mathbf u}\rangle_{\mathcal U_{\Delta t}}
:=\Delta t\sum_{n=1}^{N_T}(\mathbf u^n)^\top\widetilde{\mathbf u}^{\,n},
\qquad
\|\mathbf u\|_{\mathcal U_{\Delta t}}
:=\langle\mathbf u,\mathbf u\rangle_{\mathcal U_{\Delta t}}^{1/2}
=\left(\Delta t\sum_{n=1}^{N_T}|\mathbf u^n|_2^2\right)^{1/2}.
\end{equation}
We use this control space in both the full-order and reduced-order optimization, since the control dimension is unchanged by POD reduction. With respect to this inner product, the gradient is

\begin{equation}\label{eq:discrete-gradient}
\mathbf g^n=\beta\mathbf u^n-F^\top\mathbf w^n.
\end{equation}

We use a spectral gradient method with Barzilai--Borwein step sizes \cite{azmiKunisch3,AzmiKunisch7,barzilai1988two} and a nonmonotone line search.

\section{POD-based receding horizon control}\label{sec:MOR_RHC}

In this section, we develop a MOR strategy for RHC of the translated Navier--Stokes system \eqref{e16n}. We construct a single $M$-orthonormal POD basis from full-order state snapshots and use the same reduced space for the state and adjoint equations in both numerical examples. On the first horizon $(0,T)$, we solve the finite-horizon problem numerically at the full-order FE level and collect the state iterates generated during the optimization. This basis defines the reduced models used on subsequent horizons. We compute controls with the current reduced models and apply them to the full-order system. The resulting full-order state trajectory supplies new snapshots when we update the common POD basis. More precisely, we proceed as follows:
\begin{enumerate}[(i)]
    \item Solve the finite-horizon optimal control problem on $(0,T)$ at the full-order FE level and collect the state trajectories generated during the optimization.
    \item Compute an $M$-orthonormal POD basis $\Psi_\bv\in\mathbb R^{n_v\times\ell}$, with $\ell\ll n_v$, from these state snapshots.
    \item Construct the state ROM in this POD space and derive its discrete adjoint using the same space.
    \item On each subsequent horizon $(t_k,t_k+T)$, compute an approximate stationary point of the discretized reduced optimal control problem and apply it to the full-order system.
    \item Collect the resulting full-order state trajectory as new snapshot data.
    \item When \texttt{updateBasis = True}, update the common POD basis with the new state snapshots and reconstruct the reduced models for the next RHC step.
\end{enumerate}

\subsection{Abstract POD Formulation}
\label{subsec:POD-abstract}
We consider the discrete Hilbert space $\R^{n_v}$ and let
$z_1,\dots,z_{N_s}\in\R^{n_v}$ denote a set of snapshots, i.e., time-discrete
state snapshots, indexed over the selected time levels and optimization iterations. We further introduce positive weights
$\alpha_1,\dots,\alpha_{N_s}>0$, which in practice are chosen according to a
quadrature rule in time, e.g., the trapezoidal rule. For a prescribed reduced
dimension $1\leq\ell\leq d_v$, the discrete POD problem reads
\begin{equation}
\label{eq:POD-abstract}
\begin{aligned}
 \min_{\psi_1,\dots,\psi_\ell\in \R^{n_v}}\;
   &\frac12 \sum_{k=1}^{N_s} \alpha_k
   \biggl\| z_k-\sum_{i=1}^\ell \langle z_k,\psi_i\rangle \psi_i \biggr\|^2 \\
 \text{s.t.}\qquad
   &\langle \psi_i,\psi_j\rangle=\delta_{ij},
   \qquad i,j=1,\dots,\ell,
\end{aligned}
\end{equation}
where $\delta_{ij}$ denotes the Kronecker delta. The inner product
$\langle\cdot,\cdot\rangle$ will be specified in Section~\ref{sec.discPOD}. The pairwise
orthonormal vectors $\psi_1,\dots,\psi_\ell$ are called POD modes of rank
$\ell$. The solution of \eqref{eq:POD-abstract} is characterized by the correlation operator
\begin{equation}
  R:\R^{n_v}\to \R^{n_v},
  \qquad
  R\psi := \sum_{k=1}^{N_s} \alpha_k \langle z_k,\psi\rangle z_k.
\end{equation}
The operator $R$ is linear, self-adjoint, and positive semidefinite; see,
e.g., \cite{gubisch2017proper}. Let $\{(\lambda_i,\varphi_i)\}_{i=1}^{d_v}$
denote the nonzero eigenpairs of $R$, with eigenvectors orthonormal in the chosen inner product, ordered such that $\lambda_1 \ge \lambda_2 \ge \cdots > 0$, where $d_v:=\dim\bigl(\mathrm{span}\{z_1,\dots,z_{N_s}\}\bigr)\leq N_s$.
Then the POD modes of rank $\ell$ are given by $\psi_i=\varphi_i$ for $i=1,\dots,\ell$.

\subsection{Snapshot Generation for Navier--Stokes RHC}

We generate the initial snapshots by solving the full-order optimal control problem on the first prediction horizon $(0,T)$, starting from $\bv_0$. We use the FE discretization and spectral gradient method described in Section~\ref{sec:disc_and_impl}. Each optimization iteration requires a forward solve of \eqref{eq:v_disc} and a backward solve of \eqref{eq:w_disc}. We collect the time-discrete state iterates to construct the POD basis. Using these iterates enriches the snapshot space with states encountered during the optimization. Denoting the state at time $t_n$ in iteration $i$ by $\bv_n^{(i)}$, we assemble
\[
\mathbf V:=[\bv_0^{(1)},\ldots,\bv_{N_T}^{(1)},
\bv_0^{(2)},\ldots,\bv_{N_T}^{(2)},\ldots]
\in\mathbb R^{n_v\times N_s},
\]
where $N_s$ is the total number of stored state snapshots. We use the resulting POD basis for both the reduced state and adjoint. When we update the basis, we append new full-order state snapshots to $\mathbf V$ and repeat the same POD construction.

\subsection{Discrete POD}
\label{sec.discPOD}

The FE discretization induces the inner product
\[
\langle z,v\rangle_M := z^\top Mv,
\qquad
\|z\|_M^2 := z^\top M z,
\]
on $\R^{n_v}$, where $M$ is the mass matrix introduced in
\eqref{eq:fe_mat}. We use this inner product in the discrete POD problem \eqref{eq:POD-abstract} for the state snapshots. Moreover, let $D_s:=\diag(\alpha_1,\dots,\alpha_{N_s})\in\mathbb R^{N_s\times N_s}$
denote the diagonal matrix of quadrature weights. For the state snapshots, we introduce the weighted snapshot matrix $\widehat{\mathbf V}:=M^{1/2}\mathbf VD_s^{1/2}$, and compute its singular value decomposition $\widehat{\mathbf V}=U_\bv \Sigma_\bv V_\bv^\top$, where $U_\bv\in\R^{n_v\times d_v}$ has orthonormal columns, $\Sigma_\bv=\diag(\sigma_1^\bv,\dots,\sigma_{d_v}^\bv)$ contains the nonzero singular values, and $d_v=\mathrm{rank}(\widehat{\mathbf V})\leq \min\{n_v,N_s\}$. The POD basis matrix for the state is then defined by
\[
\Psi_\bv := M^{-1/2}U_\bv^\ell \in \R^{n_v\times \ell},
\]
where $U_\bv^\ell$ denotes the first $\ell$ columns of $U_\bv$. By construction, the columns of $\Psi_\bv$ are $M$-orthonormal. We use this common basis for both the state and adjoint. A POD space generated from state snapshots is also used for both variables in \cite[Remark 5.2]{kammann2013posteriori}. We next derive the Galerkin ROM for the translated Navier--Stokes system. \par
We retain the full FE reference in the Galerkin projection. The reduced state equation takes the form
\begin{equation}\label{eq:rom-full-reference}
M^\ell\partial_t\mathbf v^\ell+\nu A^\ell\mathbf v^\ell+
O^\ell(t)\mathbf v^\ell+C^\ell(\mathbf v^\ell\otimes\mathbf v^\ell)=F^\ell\mathbf u,
\end{equation}
where
\begin{align*}
M^\ell&:=\Psi_\bv^\top M\Psi_\bv\in\mathbb R^{\ell\times\ell},\qquad
A^\ell:=\Psi_\bv^\top A\Psi_\bv\in\mathbb R^{\ell\times\ell},\qquad
O^\ell(t):=\Psi_\bv^\top O(t)\Psi_\bv,\\
F^\ell&:=\Psi_\bv^\top F\in\mathbb R^{\ell\times N},\qquad
C^\ell:=\Psi_\bv^\top C(\Psi_\bv\otimes\Psi_\bv)\in\mathbb R^{\ell\times\ell^2},
\end{align*}
and $\mathbf v^\ell(t)\in\mathbb R^\ell$ contains the reduced state coefficients. For a stationary reference, we assemble $O^\ell$ once per basis update; for a time-dependent reference, we evaluate it at the time levels used in the state scheme.

\paragraph{Reduced adjoint in the common POD space.}
We keep $\Psi_\bv$ fixed throughout each finite-horizon optimization and represent the reduced state and adjoint in its span. For reduced state coefficients $\mathbf v^{\ell,n}$, we define
\[
G_n^\ell(\mathbf v^\ell):=C^\ell(\mathbf v^\ell\otimes\mathbf v^\ell)+O^\ell(t_n)\mathbf v^\ell,
\qquad
DG_n^\ell(\mathbf v^\ell)[\delta\mathbf v^\ell]
:=C^\ell(\delta\mathbf v^\ell\otimes\mathbf v^\ell)+C^\ell(\mathbf v^\ell\otimes\delta\mathbf v^\ell)+O^\ell(t_n)\delta\mathbf v^\ell.
\]
We derive the reduced adjoint from the transpose of the reduced state Jacobian, using the reduced counterpart of \eqref{eq:discrete-cost} with the same time-step summation. Equivalently, we project the full-order discrete adjoint onto the common POD space, evaluated along the reconstructed reduced state. This gives
\begin{equation}\label{eq:common-pod-adjoint}
\begin{aligned}
(\Delta t^{-1}M^\ell+\nu A^\ell)^\top\mathbf w^{\ell,N_T}
&=-M^\ell\mathbf v^{\ell,N_T},\\
(\Delta t^{-1}M^\ell+\nu A^\ell)^\top\mathbf w^{\ell,n}
&=\bigl(\Delta t^{-1}M^\ell-DG_n^\ell(\mathbf v^{\ell,n})\bigr)^\top\mathbf w^{\ell,n+1}-M^\ell\mathbf v^{\ell,n},
\quad n=N_T-1,\ldots,1.
\end{aligned}
\end{equation}
Here $\mathbf w^{\ell,n}\in\mathbb R^\ell$ contains the adjoint coefficients in the common POD basis, so the corresponding FE vector is $\Psi_\bv\mathbf w^{\ell,n}$. With the control-space inner product \eqref{eq:discrete-control-space}, the reduced gradient is

\begin{equation}\label{eq:common-pod-gradient}
\mathbf g_\ell^n=\beta\mathbf u^n-(F^\ell)^\top\mathbf w^{\ell,n},
\qquad n=1,\ldots,N_T.
\end{equation}
We use this common-space formulation in both numerical examples.

 Finally, the state snapshots contained in $\mathbf V$ are discretely divergence-free by construction. Consequently, the POD basis functions inherit this property, and the pressure term vanishes in the reduced system; see \cite{gunzburger2017ensemble}.

With the full-order and reduced-order discretizations in place, we are now in a position to state the MOR-based RHC algorithm; see Algorithm~\ref{alg:FOM-ROM_RHC}. In this computational algorithm, a superscript $*$ denotes an approximate stationary point of the discretized optimization problem.

Similarly to \eqref{OpT_tran}, on each prediction horizon we define the reduced optimal control problem
\begin{equation}\label{eq:rom-ocp}
\tag*{\ensuremath{\mathrm{OP}^{\ell}_{T}(\bar{t}_0,\bar{\mathbf{v}}_0^\ell)}}
\min_{\bu\in L^2(I_T(\bar t_0);\mathbb R^N)}
\{J_T^\ell(\bu;\bar t_0,\bar\bv_0^\ell)\text{ s.t. }\eqref{eq:rom-full-reference}\text{ and }\bv^\ell(\bar t_0)=\bar\bv_0^\ell\}.
\end{equation}
Here the objective function is defined by
\begin{equation}\label{eq:rom-objective}
J_T^\ell(\bu;\bar t_0,\bar\bv_0^\ell)
:=\frac12\int_{\bar t_0}^{\bar t_0+T}\|\bv^\ell(t)\|_{M^\ell}^2\,dt
+\frac\beta2\int_{\bar t_0}^{\bar t_0+T}|\bu(t)|_2^2\,dt,
\end{equation}
and in Algorithm~\ref{alg:FOM-ROM_RHC}, we solve the time-discrete counterpart of \eqref{eq:rom-ocp} numerically.

\begin{algorithm}[!htbp]
\caption{MOR-based RHC($\delta,T,\texttt{updateBasis}$)}\label{alg:FOM-ROM_RHC}
\begin{algorithmic}[1]
\REQUIRE{Final time $T_{\infty} \in \mathbb{R}_{\geq 0}$, sampling time $\delta$, the prediction horizon $T\geq \delta$, 
a reference trajectory $\hat{\mathbf{y}}$, the initial state $\mathbf{v}_0 \coloneqq \mathbf{y}_0-\hat{\mathbf{y}}_0$, and a Boolean flag \texttt{updateBasis} indicating whether the common POD basis is updated during RHC.}
\ENSURE{The MOR-based RHC~$\mathbf{u}_{mrh}$, nondecreasing sequence~$\{ t_i \}_{i\in\mathbb{N}}$.}
\STATE Set $\bar t_0 \gets 0$, $\bar\bv_0 \gets \bv_0$, and $i\gets 0$.
\STATE Solve the full-order problem \eqref{OpT_tran} on $I_T(\bar t_0)$ to obtain the approximate stationary pair  $(\mathbf{v}_T^*(\cdot;\bar{t}_0, \bar{\mathbf{v}}_0), \mathbf{u}^*_T(\cdot;\bar{t}_0, \bar{\mathbf{v}}_0))$ and the associated adjoint $\bw_T^*(\cdot;\bar t_0,\bar\bv_0)$.
\STATE Construct the initial state snapshot matrix $\mathbf V$ from the full-order optimization iterates and compute the common POD basis $\Psi_\bv$ for the state and adjoint.
\STATE Set $t_{i+1}\gets \min\{\bar t_0+\delta,T_\infty\}$.
\STATE For all $\tau\in[\bar{t}_0,t_{i+1})$ set $\mathbf{v}_{mrh}(\tau):= \mathbf{v}_T^*(\tau;\bar{t}_0,\bar{\mathbf{v}}_0)$ and $\mathbf{u}_{mrh}(\tau):=\mathbf{u}^*_T(\tau;\bar{t}_0,\bar{\mathbf{v}}_0)$.
\STATE Update $\bar\bv_0 \gets \bv_T^*(t_{i+1};\bar t_0,\bar\bv_0)$, $\bar t_0\gets t_{i+1}$, and $i\gets i+1$.
\WHILE{$ {\bar{t}_{0}} <T_\infty$}
\STATE Project the current FOM state $\bar\bv_0$ onto  the reduced space, i.e., $\bar\bv^\ell_0 := \Psi_\bv^\top M \bar\bv_0$;
\STATE Compute an approximate stationary point $(\mathbf{v}_T^{\ell,*}(\cdot;\bar{t}_0, \bar\bv^\ell_0), \mathbf{u}^{\ell, *}_T(\cdot;\bar{t}_0, \bar\bv^\ell_0))$ of the time-discrete reduced problem \eqref{eq:rom-ocp}.
 \STATE Compute $\bv_T(\cdot;\bar t_0,\bar\bv_0,\bu_T^{\ell,*})$ by solving the full-order translated system \eqref{e16n} on $I_T(\bar t_0)$ with initial condition $\bar\bv_0$ and control $\bu_T^{\ell,*}(\cdot;\bar t_0,\bar\bv^\ell_0)$.
\STATE Set $t_{i+1}\gets \min\{\bar t_0+\delta,T_\infty\}$.
\STATE For all $\tau\in[\bar t_0,t_{i+1})$, set $\bu_{\rm mrh}(\tau):=\bu_T^{\ell,*}(\tau;\bar t_0,\bar\bv^\ell_0)$ and $\bv_{\rm mrh}(\tau):=\bv_T(\tau;\bar t_0,\bar\bv_0,\bu_T^{\ell,*})$.
     \IF{\texttt{updateBasis}}
        \STATE Append the newly computed full-order state snapshots from $\bv_T(\cdot;\bar t_0,\bar\bv_0,\bu_T^{\ell,*})$ to $\mathbf V$.
        \STATE Recompute the common POD basis $\Psi_\bv$ from the updated state snapshot matrix and reconstruct the reduced state and adjoint models.
    \ENDIF
    
    \STATE Update $\bar\bv_0 \gets \bv_T(t_{i+1};\bar t_0,\bar\bv_0,\bu_T^{\ell,*})$, $\bar t_0\gets t_{i+1}$, and $i\gets i+1$.
\ENDWHILE
\end{algorithmic}
\end{algorithm}

As written, Algorithm~\ref{alg:FOM-ROM_RHC} includes a full-order state solve on each complete prediction horizon after computing a reduced control. Only the first sampling segment is applied in closed loop; the remaining predicted states can supply snapshots.

\section{Numerical Results}\label{sec:Numerical_Results}
We report numerical results for Algorithms~\ref{RHA2} and \ref{alg:FOM-ROM_RHC} on two examples of increasing complexity. We use the time step $\Delta t=10^{-3}$ in all numerical simulations. The finite-horizon optimal control problems are solved by the spectral gradient method from \cite{AzmiKunisch7} and terminated when the gradient norm in $\mathcal U_{\Delta t}$ falls below $\epsilon=10^{-6}$. We use the full-order gradient \eqref{eq:discrete-gradient} or the reduced gradient \eqref{eq:common-pod-gradient}, respectively, with the same norm \eqref{eq:discrete-control-space}.

For the numerical implementation, we use the componentwise actuator family \eqref{e128} introduced in Section~\ref{sec:stabilization}. Each support $R_j$ carries the two profiles $(\phi_j,0)$ and $(0,\phi_j)$, with independent scalar control coefficients. Thus $M_a$ supports give $N=2M_a$ scalar control components, as in the theoretical construction.

\begin{revision}
\end{revision}
\FloatBarrier
\subsection{Example 1: {Rotating Flow within a} Disc} In this example, we consider the stabilization of a rotating flow by means of finitely many localized actuators in the presence of nonhomogeneous Dirichlet boundary conditions. The spatial domain is taken as the unit disk with a square hole centered at the origin, namely $\Omega = B_1(0,0)\setminus [-0.25,0.25]^2$, where $B_r(x)$ denotes the open ball of radius $r>0$ centered at $x\in\R^2$. Accordingly, the boundary is decomposed as $\Gamma=\Gamma_{\mathrm{disc}}\cup \Gamma_{\mathrm{rect}}$,
where $\Gamma_{\mathrm{disc}}=\partial B_1(0,0)$ and $\Gamma_{\mathrm{rect}}=\partial[-0.25,0.25]^2$. The system is driven by the Dirichlet boundary condition $\by|_{\partial\Omega}=\hat{\bh}$, where
\[
\hat{\bh}(x_1,x_2)=
\begin{cases}
(x_2,-x_1), & x\in \Gamma_{\mathrm{disc}},\\[0.3em]
(0,0), & x\in \Gamma_{\mathrm{rect}}.
\end{cases}
\]
We set the external forcing $\hat{\mathbf f}=0$ and take $\by_0=0$ as initial condition. The target flow $\hat{\by}$ is chosen as the stationary limit flow, computed a priori from the corresponding stationary Navier--Stokes system to~\eqref{e2}, obtained by omitting the time-derivative term. We first investigate the extent to which the full-order RHC approach improves the rate of stabilization. To this end, we compare different actuator layouts, viscosities $\nu$, and prediction horizons $T$. Throughout this example, the RHC simulations are run up to the final time $T_\infty=5$ with sampling time $\delta=0.25$. We set the control penalty weight to $\beta=10^{-3}$.

\medskip
\noindent\textbf{Actuator layouts.} We begin by comparing two actuator layouts: one with 48 supports (96 scalar controls) on a square-shaped region and one with 25 supports (50 scalar controls) on an L-shaped region; see Figure~\ref{fig:L_vs_square_acts}. For $\nu=10^{-2}$ and $T=1$, the evolution of $\|\bv_{rh}(t)\|_H$ is shown in the right panel of Figure~\ref{fig:L_vs_square_acts}. In both cases, the RHC strategy significantly outperforms the uncontrolled flow. Moreover, the configuration with $48$ actuator supports achieves a noticeably faster rate of stabilization than the configuration with $25$ actuator supports.
\begin{figure}[!htbp]
    \centering
    \begin{minipage}{0.25\textwidth}
        \centering
        \includegraphics[width=\textwidth, trim={2.5cm 2.5cm 2.5cm 2.5cm},clip]{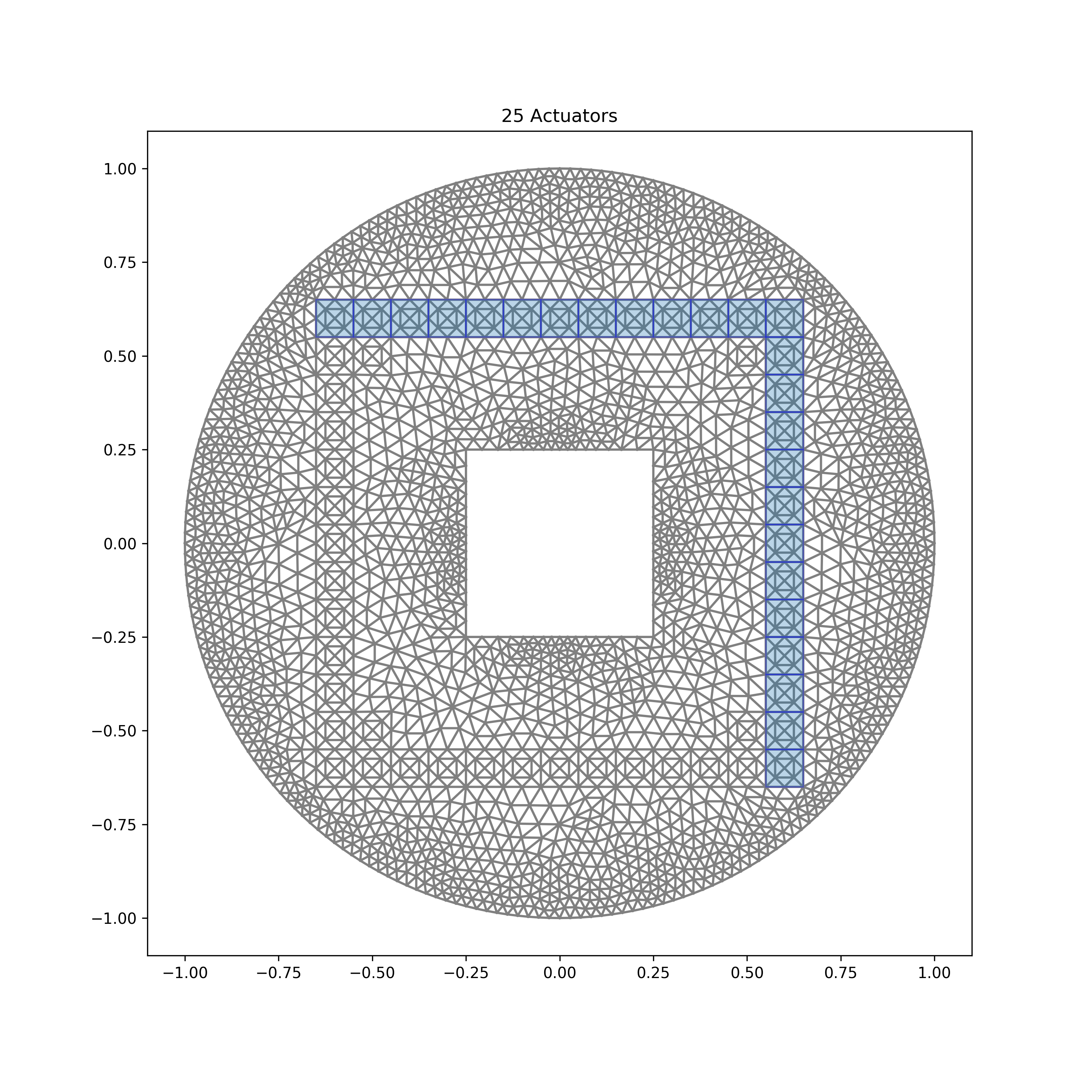}
    \end{minipage}
    \hfil
    \begin{minipage}{0.25\textwidth}
        \centering
        \includegraphics[width=\textwidth, trim={2.5cm 2.5cm 2.5cm 2.5cm}, clip]{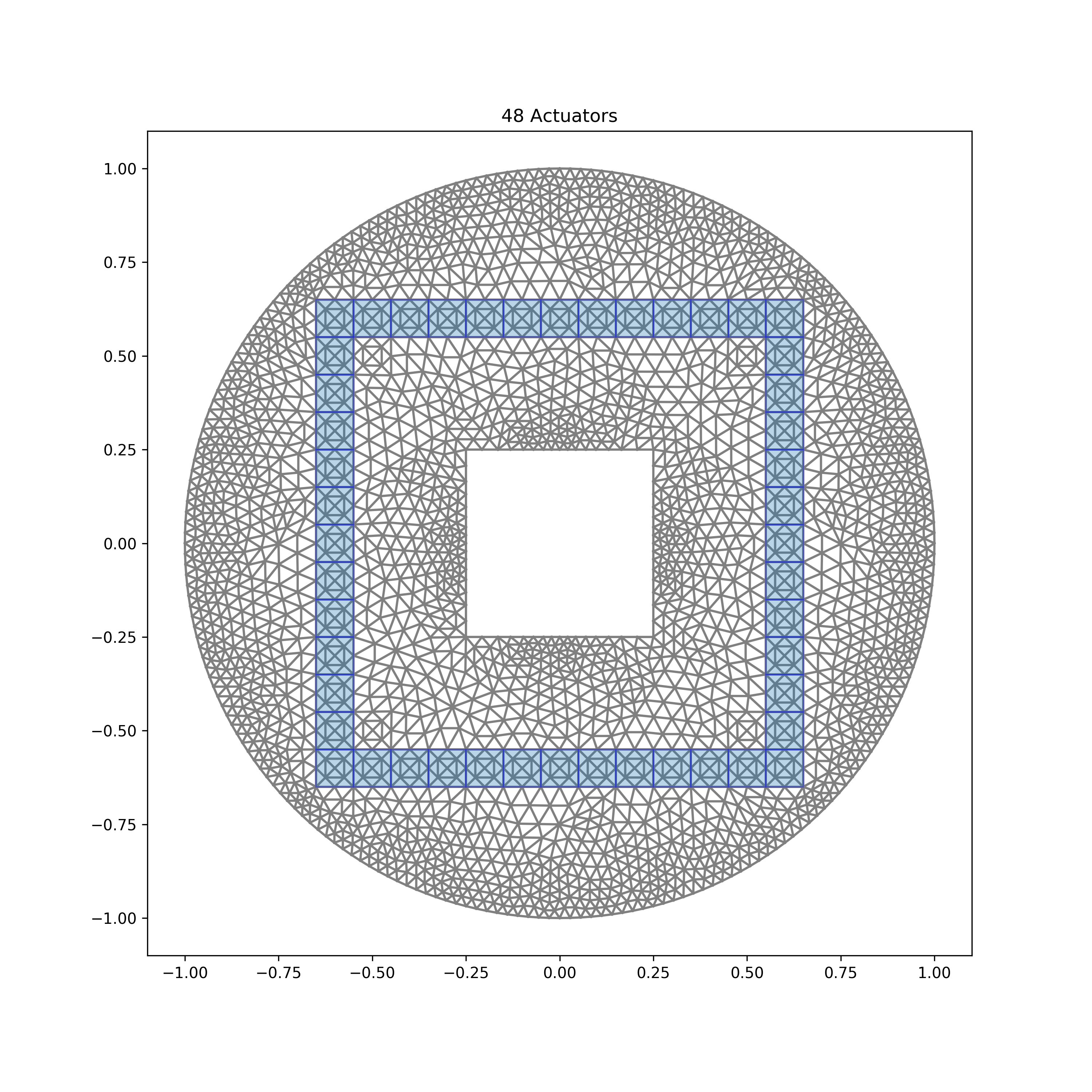}
    \end{minipage} \hfil 
    \begin{minipage}{0.4\textwidth}
        \centering
        \includegraphics[width=\textwidth]{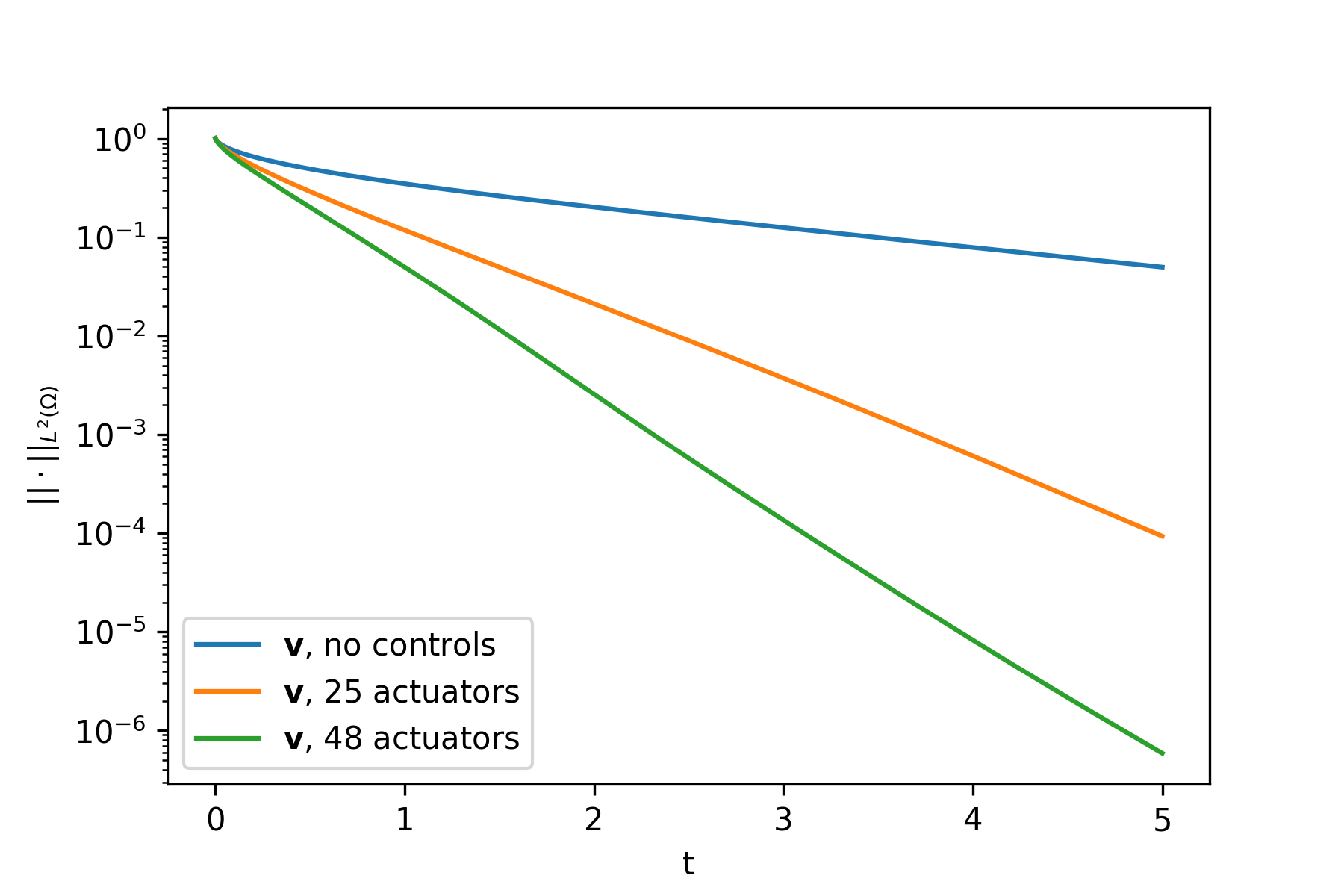}
        \end{minipage}
    \caption{\label{fig:L_vs_square_acts} \textit{Left:} Different actuator layouts used in Example 1. \textit{Right:} Evolution of $\|\mathbf{v}(t)\|_H$ for the two layouts and $T= 1$.}
\end{figure}

\medskip
\noindent\textbf{Prediction horizon.} We next fix the square actuator layout and compare the performance of the RHC scheme for the prediction horizons $T\in\{0.25,1,2\}$. In this experiment longer horizons yield faster decay and smaller observed values of $J_{T_\infty}^o$. A longer horizon includes more future dynamics. On the other hand, larger prediction horizons increase the computational cost of solving the finite-horizon subproblems. Figure~\ref{fig:RHC_different_T} illustrates the effect of the prediction horizon on the evolution of $\|\bv_{rh}(t)\|_H$ and on the values of the objective functional $J_{T_\infty}^o$ defined in \eqref{e49}. Indeed, a substantial improvement is observed when increasing the horizon from $T=0.25$ to $T=1$, whereas the additional gain obtained by further increasing the horizon from $T=1$ to $T=2$ is comparatively small.
\begin{figure}[!htbp]
    \centering
    \begin{minipage}{0.5\textwidth}
        \centering
        \includegraphics[width=\textwidth]{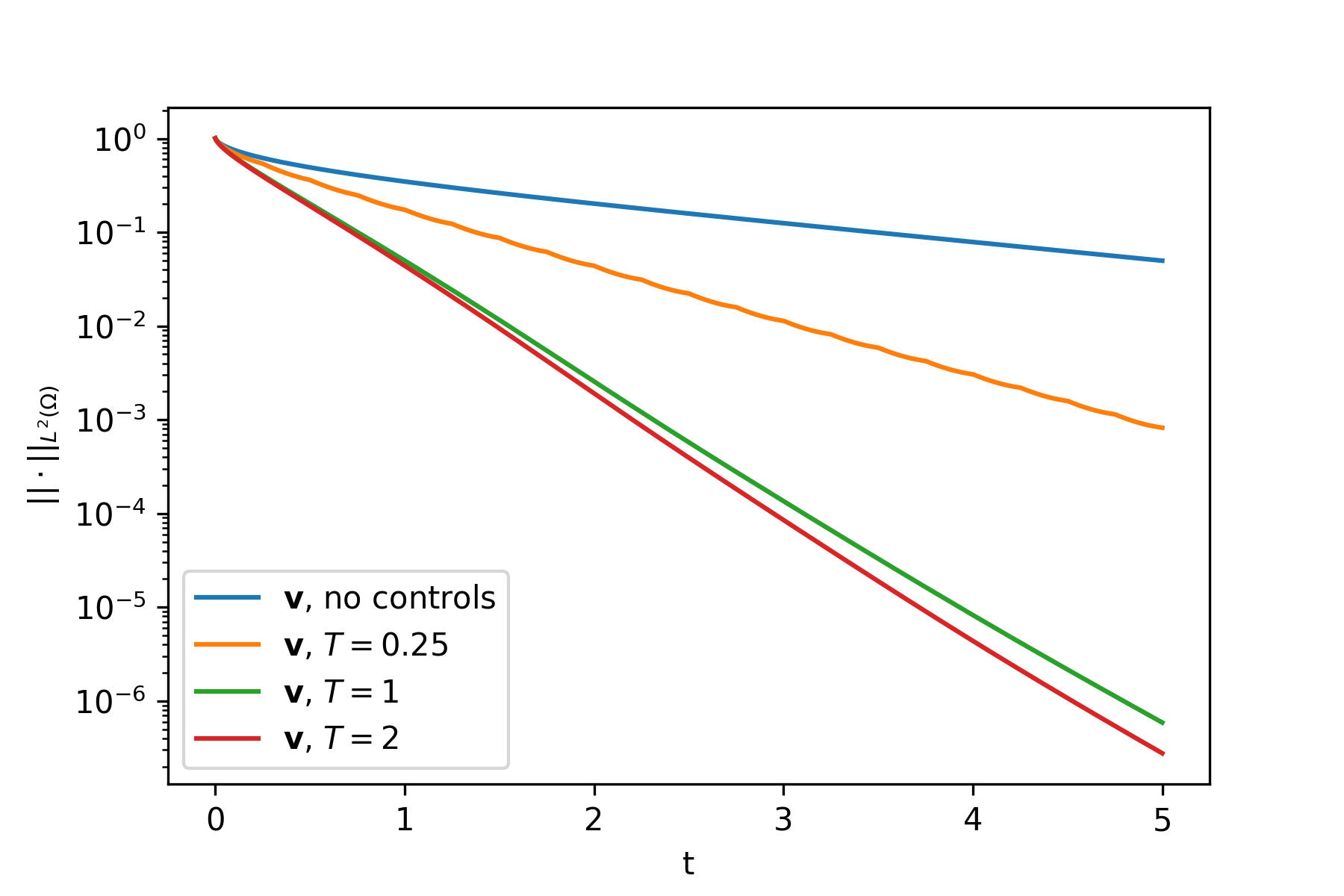}
    \end{minipage}\hfil
\begin{minipage}{0.45\textwidth}
\centering
    \renewcommand{\arraystretch}{1.5} 
    \setlength{\tabcolsep}{10pt} 
    \begin{tabular}{|c|c|}
        \hline
        $T$ & $J_{T_\infty}^o(\bu_{rh}; 0, \by_0)$ \\ 
        \hline \hline
        0.25 & 0.276436
 \\ 
        \hline
        1.0 & 0.194958
 \\ 
        \hline
        2.0 & 0.194548 \\ 
        \hline\hline
        no controls & 0.559781\\
        \hline
    \end{tabular}
    \end{minipage}
    \caption{Left: Evolution of $\|\mathbf{v}_{rh}(t)\|_H$ for RHC  with different prediction horizons. Right: Corresponding values of the cost function}\label{fig:RHC_different_T}
\end{figure}

\medskip
\noindent\textbf{Viscosity.} We next fix $T=1$ and investigate the performance of the RHC scheme for the viscosity values $\nu\in\{10^{-1},10^{-2},10^{-3}\}$. For larger viscosities, the Navier--Stokes system exhibits stronger intrinsic dissipation, and the uncontrolled flow already approaches the stationary state relatively quickly. In this regime, the improvement obtained by RHC over the uncontrolled dynamics is therefore expected to be moderate. For smaller viscosities, by contrast, the dynamics are significantly less dissipative, and the uncontrolled flow requires substantially more time to approach the stationary state. In this case, one expects the benefit of RHC to be more pronounced, as reflected in a faster decay of $\|\bv(t)\|_H$; see Figure~\ref{fig:stabilisation_disc_viscosities}. The numerical results confirm this behavior. For $\nu=0.1$, the RHC-controlled trajectory performs only slightly better than the uncontrolled one, whereas for smaller viscosities the advantage of RHC becomes clearly visible. Among the controlled trajectories, the rate of stabilization is faster for larger viscosities.

\begin{figure}[!htbp]
\centering
\includegraphics[width=0.32\textwidth]{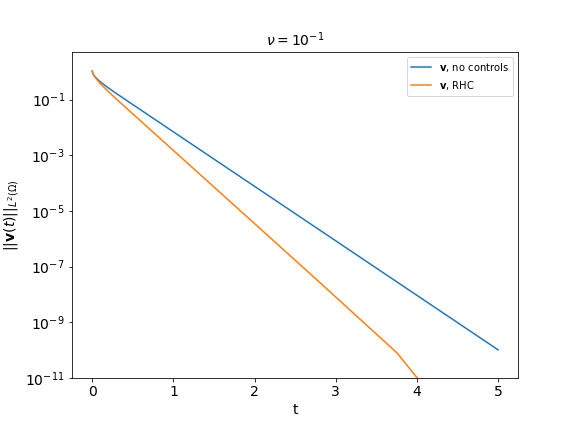}\hfill
\includegraphics[width=0.32\textwidth]{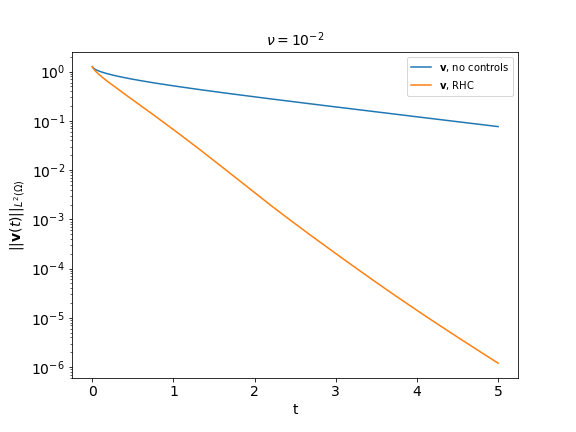}\hfill
\includegraphics[width=0.32\textwidth]{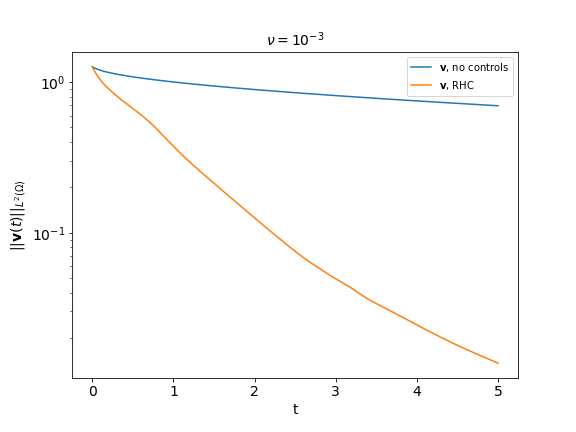}
\caption{Evolution of $\|\mathbf v(t)\|_H$ in Example 1 for $\nu=10^{-1},10^{-2},10^{-3}$ (left, middle, and right, respectively).}
\label{fig:stabilisation_disc_viscosities}
\end{figure}

\medskip
\noindent\textbf{MOR-based RHC.} In the following figures and runtime tables, FOM-ROM refers to the MOR-based RHC scheme, where the controls are computed using a POD reduced-order model and then applied to the full-order discretized dynamics. Finally, we assess the performance of the MOR-based RHC summarized in Algorithm~\ref{alg:FOM-ROM_RHC}. We fix $\nu=10^{-2}$, $T=1$, and $\delta=0.25$, and compare the method for different POD dimensions $\ell\in\{20,50,100\}$; see the left panel of Figure~\ref{fig.poddisc}. In these experiments, we set \texttt{updateBasis = False} in Algorithm~\ref{alg:FOM-ROM_RHC}, that is, the common POD basis is computed only once from the state snapshots generated during the initial full-order optimization on the first prediction horizon. The results show that the reduced-order scheme performs nearly as well as the full-order RHC method. In fact, already $\ell=20$ POD basis functions appear to be sufficient in this example, since increasing the basis size does not lead to a visible improvement. In the right panel of Figure~\ref{fig.poddisc}, we also report the runtimes of a full FOM-based RHC simulation and of the MOR-based RHC scheme. The reported total runtimes, including the first prediction horizon and the POD assembly cost, give a speedup of $45196/3382\approx13.36$.

\begin{figure}[!htbp]
\begin{minipage}{0.5\textwidth}
    \centering
    \includegraphics[width=\textwidth]{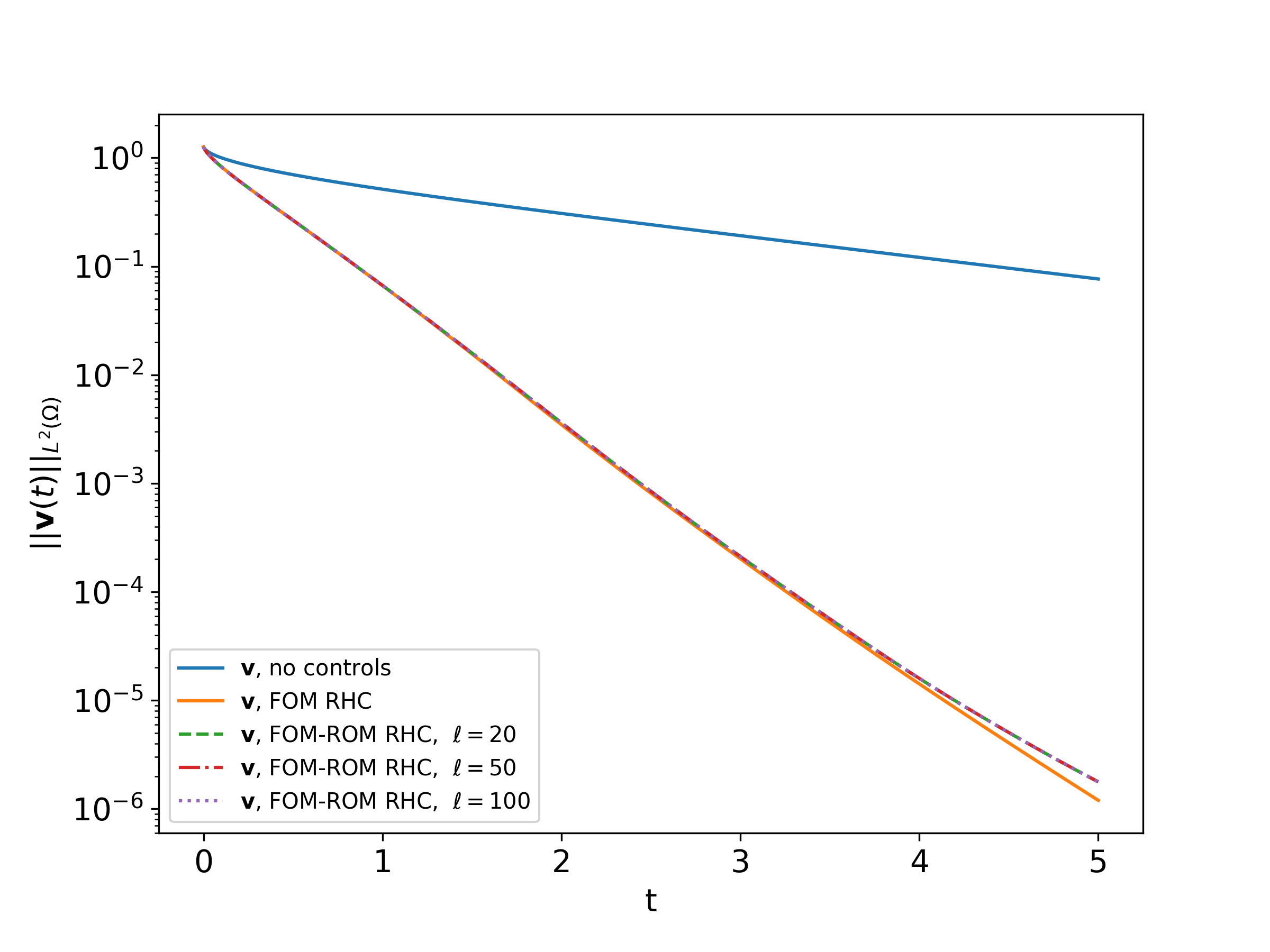}
    \end{minipage}\hfil
\begin{minipage}{0.5\textwidth}
\begin{tabular}{r|rr}
& FOM & FOM-ROM\\
\hline
First horizon & 3\,005 &3\,005\\ 
POD assembly &-- &301\\
Other horizons&42\,191 &76 \\
\hline
Total  & 45\,196 &3\,382\\
\end{tabular}
\end{minipage}
\caption{\label{fig.poddisc} Left: Evolution of  $\|\mathbf{v}(t)\|_H$  for the FOM and MOR-based RHC with different POD-basis sizes applied to Example 1. Right: Runtime (in s) of FOM-RHC Algorithm~\ref{RHA2} and the MOR-based algorithm~\ref{alg:FOM-ROM_RHC} with $\ell=20$ for Example 1.}
\end{figure}

\subsection{Example 2: Flow around an Obstacle} In this example, we investigate the performance of Algorithm~\ref{alg:FOM-ROM_RHC} in a more challenging setting, namely, when the uncontrolled unsteady flow is not asymptotically stable with respect to a given stationary reference state. To this end, we consider the classical two-dimensional channel flow past a cylinder; see Figure~\ref{fig:channel_acts} for the computational domain, its discretization, and the placement of the $57$ rectangular actuator supports ($114$ scalar controls), and \cite{schafer1996benchmark} for further details. The domain is given by $\Omega := (0,2.2)\times(0,0.41)\setminus \overline B_{0.05}(0.2,0.2)$. We decompose the boundary as $\partial\Omega=\Gamma_{\mathrm{in}}\cup\Gamma_{\mathrm{out}}\cup\Gamma_{\mathrm{walls}}\cup\Gamma_{\mathrm{cyl}}$,
with
\[
\Gamma_{\mathrm{in}}:=\{0\}\times[0,0.41],\qquad
\Gamma_{\mathrm{out}}:=\{2.2\}\times[0,0.41],
\]
\[
\Gamma_{\mathrm{walls}}:=(0,2.2)\times\{0,0.41\},\qquad
\Gamma_{\mathrm{cyl}}:=\partial B_{0.05}(0.2,0.2).
\]
Homogeneous Dirichlet boundary conditions are imposed on $\Gamma_{\mathrm{walls}}\cup\Gamma_{\mathrm{cyl}}$. On $\Gamma_{\mathrm{in}}$, we prescribe a parabolic inflow profile with maximum velocity $y_{\max}=1.5$, while on $\Gamma_{\mathrm{out}}$ we impose a do-nothing outflow condition:
\begin{equation}\label{eq:channel_flow_inflow}
    \by(0,x_2)=\left(\frac{4y_{\max}x_2(0.41-x_2)}{0.41^2},\,0\right)
    \quad \text{on }\Gamma_{\mathrm{in}},
    \qquad
    \nu \frac{\partial \by}{\partial n}-pn=0
    \quad \text{on }\Gamma_{\mathrm{out}}.
\end{equation}

We fix the viscosity at $\nu=\frac{1}{750}$, corresponding to a Reynolds number of $Re=75$. We set the control penalty weight to $\beta=10^{-4}$. For this Reynolds number, the flow develops an unsteady regime with time-periodic vortex shedding behind the obstacle; see Figure~\ref{fig:instat_stat_flow} at the top. As desired trajectory $\hat{\by}$, we again choose the stationary solution of the Navier--Stokes equations; see Figure~\ref{fig:instat_stat_flow} at the bottom. As initial condition, we take $\by_0(x)=\widetilde{\by}(x,0)$, where $\widetilde{\by}$ denotes the corresponding uncontrolled unsteady flow.

\begin{figure}[!htbp]
    \centering
    \includegraphics[width=0.86\textwidth]{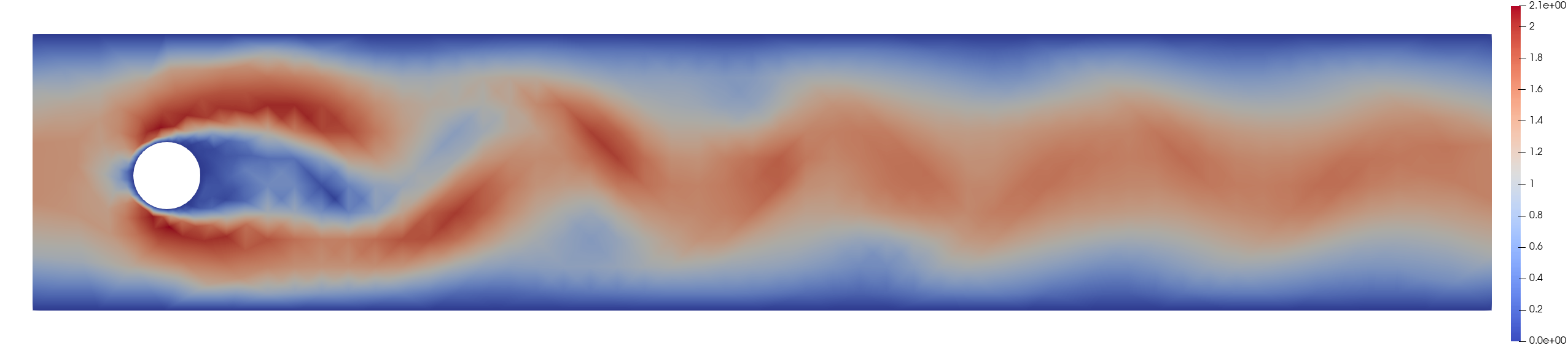}\par\vspace{2mm}
    \includegraphics[width=0.86\textwidth]{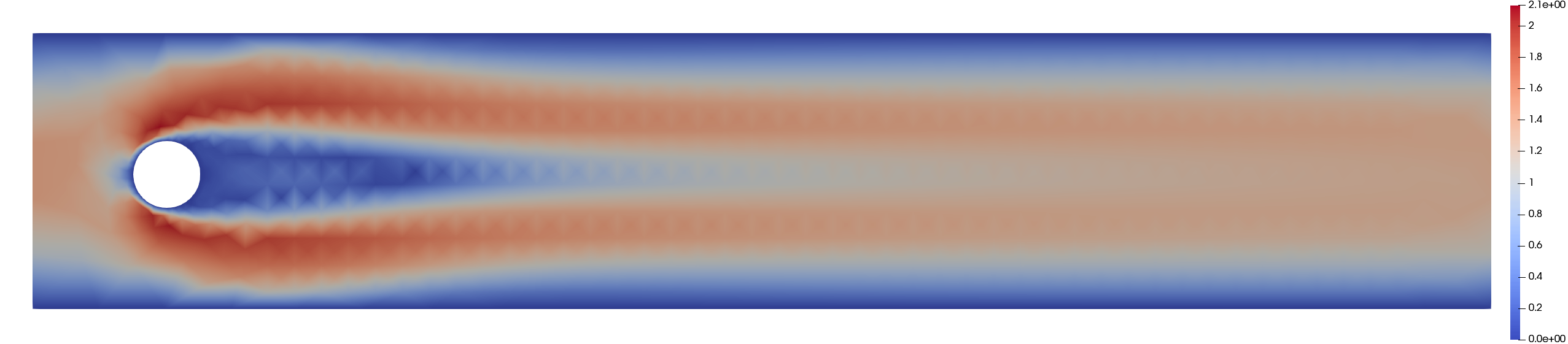}
    \caption{Top: Uncontrolled nonstationary flow and initial state $\mathbf y_0$. Bottom: Stationary reference $\hat{\mathbf y}$ in Example 2.}
    \label{fig:instat_stat_flow}
\end{figure}
\begin{figure}[!htbp]
    \centering
    \includegraphics[width=0.86\textwidth]{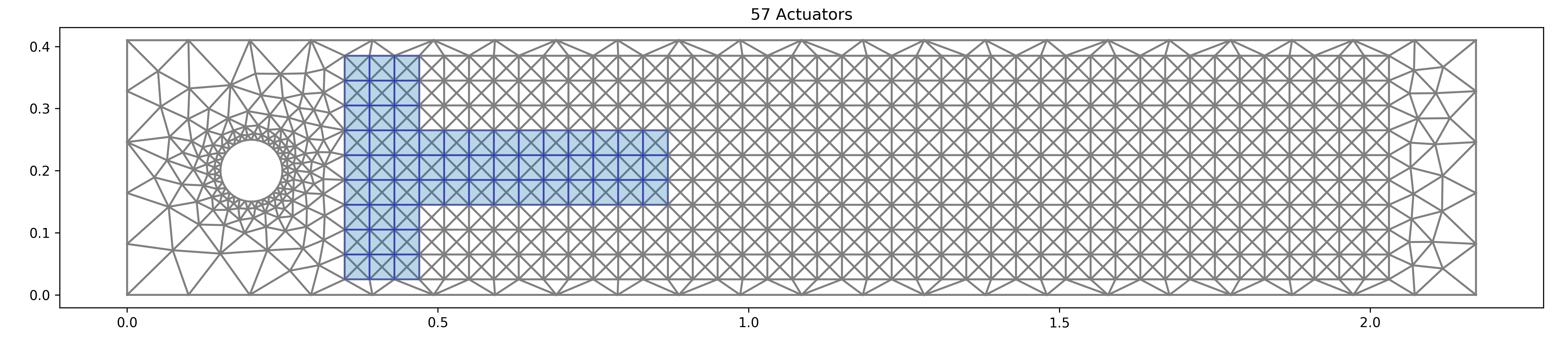}
    \caption{Spatial discretization and actuator layout for the channel flow (Example 2)}\label{fig:channel_acts}
\end{figure}

This example is not covered directly by the theoretical setting of the paper, since the boundary conditions are mixed and include a do-nothing outflow condition on $\Gamma_{\mathrm{out}}$. Nevertheless, the same numerical RHC algorithm can also be applied in this setting. Thus, Example 2 should be viewed as a robustness test of the computational RHC/MOR strategy beyond the assumptions of the theoretical analysis.

In the absence of control, the flow evolves toward the unsteady vortex-shedding regime. Our objective is to steer the system toward the stationary reference flow $\hat{\by}$ by means of finitely many actuators placed downstream of the obstacle in the T-shaped configuration shown in Figure~\ref{fig:channel_acts}, that is, to achieve $\lim_{t\to\infty}\|\by(t)-\hat{\by}\|_{\mathbf L^2}=0$.

\medskip
\noindent\textbf{Basis updates.}  As in the previous example, we first apply Algorithm~\ref{alg:FOM-ROM_RHC} with \texttt{updateBasis = False} and choose $T=0.3$, $\delta=0.25$, and $\ell=50$. As shown in Figure~\ref{fig:no_updates_fail}, the ROM quickly loses its ability to generate controls that steer the full-order state toward the stationary reference solution in this example, which is consistent with loss of accuracy of the fixed model along the later trajectory. While $\|\bv(t)\|_{\mathbf L^2}$ decreases monotonically over the first five sampling intervals $(t_0,t_5)$, separated by dotted vertical lines, this behavior changes from the following horizons onward. The controlled trajectory then deteriorates and slowly approaches the uncontrolled one. We also note that, on the first interval, the full-order RHC and MOR-based RHC solutions coincide by construction. In the remaining experiments, we therefore set \texttt{updateBasis = True}, that is, after each feedback horizon we use the newly computed full-order state snapshots to update the common POD basis for the state and adjoint.

\begin{figure}[!htbp]
    \centering
\includegraphics[width=0.64\textwidth]{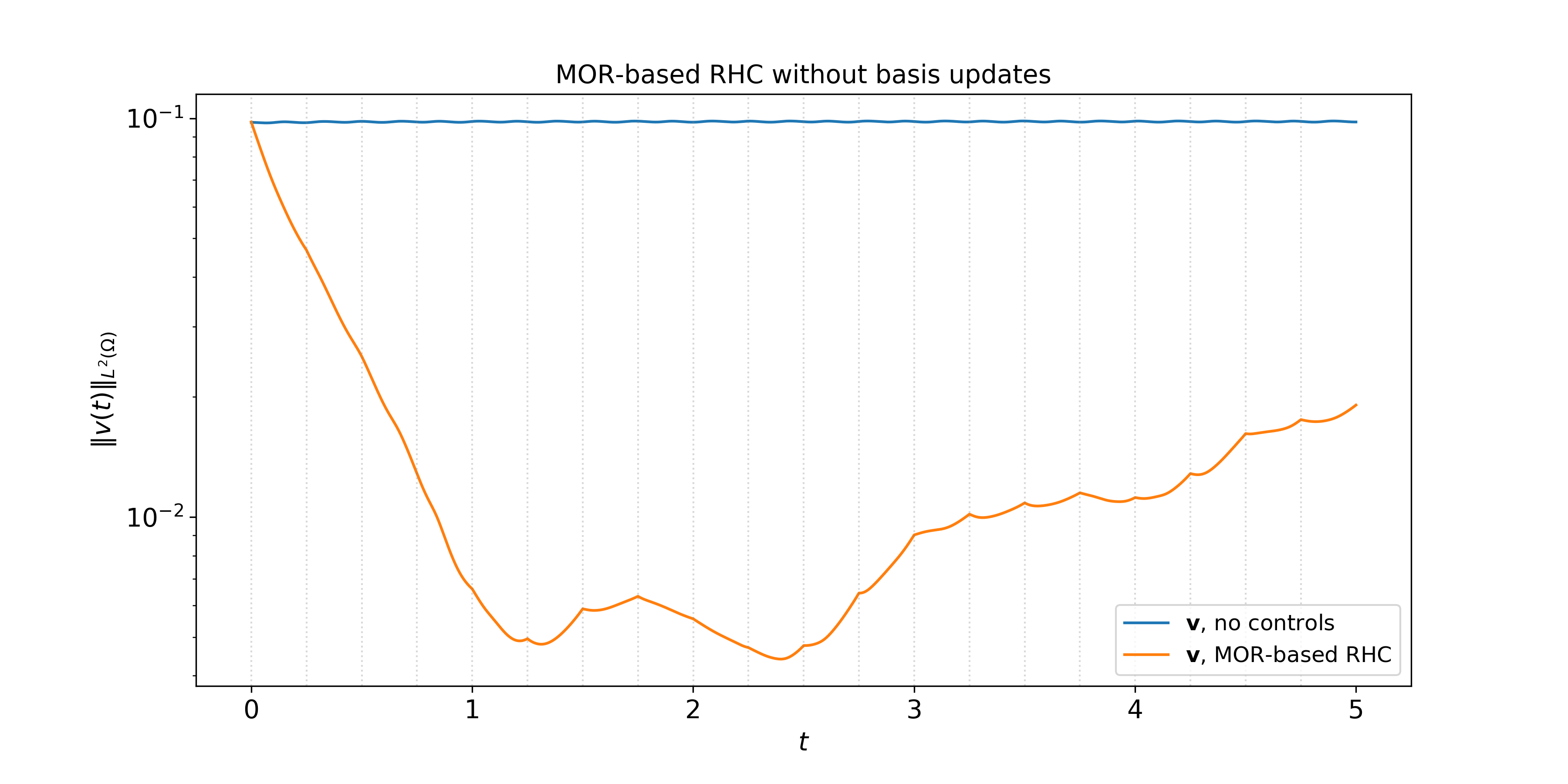}
    \caption{Evolution of $\|\mathbf v(t)\|_{\mathbf L^2}$
for the uncontrolled state and the MOR-based RHC algorithm
with $T=0.3$ and $\ell=50$, using a fixed POD basis
constructed from state snapshots generated during the
initial full-order optimization on $(0,T)$.} \label{fig:no_updates_fail}
\end{figure}

\begin{figure}[!htbp]
    \centering
    \includegraphics[
        width=0.66\textwidth,
        trim=1.5cm 0cm 3cm 0cm,
        clip
    ]{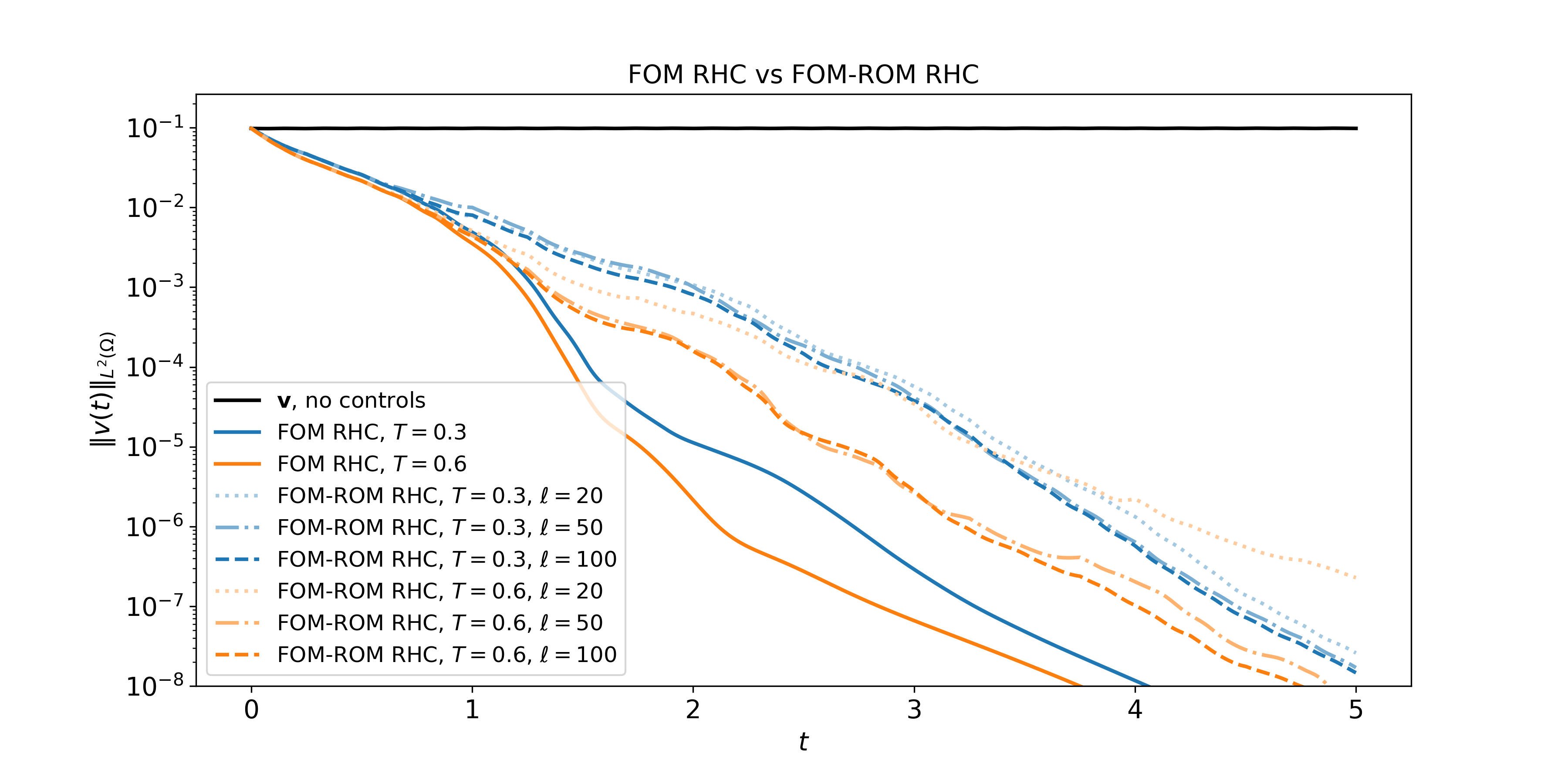}
    \caption{Evolution of $\|\mathbf{v}(t)\|_{\mathbf L^2}$ for the FOM RHC and MOR-based RHC algorithms (FOM-ROM) applied to Example 2 with $T\in\{0.3, 0.6\}$, where the ROM basis is updated regularly.}
    \label{fig:channel_FOM_vs_FOM-ROM}
\end{figure}

\begin{table}[!htbp]
\centering
\begin{tabular}{lrrrr}
\hline
Method & Total & First horizon & POD updates/assembly & Other horizons\\
\hline
FOM & 4\,654 & 441 & -- & 4\,213 \\
FOM-ROM  & 735 & 441 & 290 & 4 \\
\hline
\end{tabular}
\caption{Runtime (in s) for the FOM- and MOR-based RHCs applied to Example 2 with $T = 0.3$ and $\ell = 20.$}\label{tab:channel_runtimes}
\end{table}

\medskip
\noindent\textbf{Prediction horizon and MOR dimension.} Next, we vary the prediction horizons $T\in\{0.3,0.6\}$ and the reduced dimensions $\ell\in\{20,50,100\}$. The corresponding evolution of $\|\bv_{\rm mrh}(t)\|_{\bluechange{\mathbf L^2}}$ is shown in Figure~\ref{fig:channel_FOM_vs_FOM-ROM}. The benefit of increasing the horizon depends on reduced dimension, as discussed below; the full-order comparisons favor the longer horizon in this experiment. With basis updates included, the MOR-based RHC is now able to steer the flow toward the stationary reference state. In the reported simulations, the tracking norm under the updated MOR-based RHC decreases below $10^{-6}$. At earlier times, it can be larger than the tracking norm under full-order RHC.

Comparing the MOR-based results for different reduced dimensions, we find that enlarging the POD basis systematically improves the rate of stabilization. However, the improvement obtained when increasing $\ell$ from $20$ to $50$ is much more pronounced than the additional gain from increasing $\ell$ from $50$ to $100$. This indicates that, beyond a certain reduced dimension, the dominant flow structures are already sufficiently well resolved, such that further enrichment of the basis yields only marginal improvements. For $\ell=20$, the MOR-based RHC with $T=0.6$ performs worse than with $T=0.3$. This suggests that, for small MOR dimensions, the longer prediction horizon amplifies the effect of model reduction errors, since inaccuracies in the reduced dynamics accumulate over the horizon. The corresponding runtimes are reported in Table~\ref{tab:channel_runtimes}. Despite updating the POD basis at every RHC step, the reported total runtime gives a speedup of $4654/735\approx6.33$, including the initial full-order horizon and basis-update/assembly work. This motivates an adaptive update strategy based on an a posteriori error estimator, which could further reduce the runtime by updating the basis only when needed.

\begin{revision}
\end{revision}
\section{Concluding remarks}\label{sec:conclusions}

In this work, we have developed an efficient RHC framework for steering two-dimensional Navier--Stokes flows to prescribed reference trajectories using finitely many localized actuators. At the continuous PDE level, we established local exponential stabilizability and suboptimality for the resulting finite-dimensional actuator-based RHC scheme. Starting from the exponentially stabilizing linear feedback of \cite[Proposition 4]{Azmi5}, we treat the nonlinearity as a local perturbation. Weighted estimates for $V'$-valued forcing and a fixed-point argument yield stabilization of weak solutions for small $L^2$ initial perturbations.

To reduce the computational cost associated with repeatedly solving high-dimensional finite-horizon optimal
control problems, we combined the RHC strategy with a POD-based model order reduction approach. We constructed a common POD space from state snapshots and used it to reduce both the state and adjoint equations and accelerate the online optimization, while applying the resulting controls to the high-fidelity discretized dynamics. This
yields a practical POD-accelerated RHC method that retains the receding-horizon feedback structure
while substantially reducing the computational effort.

The numerical experiments demonstrate the effectiveness of the proposed approach for two flow configurations of increasing complexity. In the first example, a fixed POD basis computed from the initial prediction horizon was sufficient to reproduce the stabilizing behavior of the full-order RHC scheme with a significant reduction in runtime. In the second, more dynamically complex example, a fixed reduced basis was no longer sufficient over successive RHC horizons. The simulations show that basis updates recover tracking over the tested interval while retaining a substantial total runtime advantage. This highlights the importance of basis enrichment in more challenging flow regimes.

Overall, the results indicate that POD acceleration can make finite-dimensional RHC for Navier--Stokes flows computationally feasible without sacrificing favorable closed-loop stabilization behavior in the tested configurations. A natural direction for future work is the development of adaptive basis-update strategies based on a posteriori error indicators or performance estimators, which would update the reduced spaces only when necessary. Another relevant extension is the derivation of rigorous stability and suboptimality guarantees for the POD-reduced closed-loop scheme itself.

\medskip
\backmatter
\section*{Statements and Declarations}
\subsection*{Competing interests}
The authors declare that they have no relevant financial or non-financial interests to disclose.
\subsection*{Data availability}
The datasets generated during and/or analysed during the current study are available from the corresponding author on reasonable request.

\subsection*{Code availability}
Requests for access to the simulation and optimization code should be addressed to the corresponding author.

\subsection*{AI assistance}
During manuscript preparation, OpenAI Codex was used to assist with editing, clarification, and consistency checks. The authors take full responsibility for the final content, including the research ideas, theoretical results and proofs, and numerical simulations.

\begin{appendices}
\section{Proof of Theorem \ref{theo1}}
\label{proof_theo1}

\noindent\textbf{Linear feedback and notation.}
The linear feedback construction of \cite[Propositions 3 and 4]{Azmi5} also applies in two dimensions. The underlying null-controllability result \cite[Proposition 1]{MR2103189} holds in both dimensions, and the remaining steps use energy estimates and \eqref{eq:actuator-approximation}. For the prescribed rate $\lambda>0$, we set $\bar\lambda=2\lambda$ and choose $\Upsilon$ in \eqref{e34} as the approximation threshold of \cite[Proposition 4]{Azmi5} at rate $\bar\lambda$. The actuator approximation estimate in Section~\ref{sec:stabilization} shows that \eqref{eq:actuator-approximation} holds for all sufficiently large $N=2d^2$, as quantified by \eqref{e34}. Hence \cite[Proposition 4]{Azmi5} provides a weakly operator-continuous, uniformly bounded stabilizing feedback
$\mathbf K_{\bar\lambda}(t):H\to\operatorname{span}\mathcal U_\omega$.
We set $\tilde{\mathbf K}_\lambda(t):=\mathbf K_{\bar\lambda}(t)$ and define
\[
\mathcal L(t)=\nu\mathcal A+\mathcal B(\hat{\mathbf y}(t))-\Pi\tilde{\mathbf K}_\lambda(t).
\]
Then, we consider the homogeneous linear equation
\begin{equation}\label{e41}
\partial_t\mathbf v+\mathcal L(t)\mathbf v=0,\qquad \mathbf v(s)=\mathbf v_s.
\end{equation}
We denote its associated evolution family by $\mathbb W(s,t)$, so that $\mathbf v(t)=\mathbb W(s,t)\mathbf v_s$ for $t\geq s$. By \cite[Proposition 4]{Azmi5}, there exist constants $M\geq1$ and $K_0<\infty$, independent of $s,t$, such that
\begin{equation}\label{eq:linear-feedback-bounds}
\|\mathbb W(s,t)\|_{\mathcal L(H)}\leq M e^{-\bar\lambda(t-s)/2},
\qquad
\sup_{t\geq0}\|\Pi\tilde{\mathbf K}_\lambda(t)\|_{\mathcal L(H)}\leq K_0.
\end{equation}
In the following, we use \eqref{eq:linear-feedback-bounds} together with \eqref{eq:oseen-bounds}. We denote by $C>0$ a generic constant that may change from one estimate to the next and depends only on the fixed PDE data, the chosen feedback, and the rates $\lambda,\bar\lambda$. In particular, $C$ is independent of the initial time, the initial state, and the forcing. We use $C_\tau$ when the constant may additionally depend on an interval length $\tau$, and keep $C_L,C_N$ fixed in the estimates used for the contraction argument.

\medskip
\noindent\textbf{Step 1: a weighted estimate for $V'$ forcing.}
We fix an arbitrary initial time $\bar t_0\geq0$ and define
\[
X_\lambda(\bar t_0)=\left\{\mathbf w:
 e^{\lambda(\cdot-\bar t_0)/2}\mathbf w
 \in L^\infty(I_\infty(\bar t_0);H)\cap L^2(I_\infty(\bar t_0);V)\right\},
\]
with norm
\[
\|\mathbf w\|_{X_\lambda(\bar t_0)}^2
=\operatorname*{ess\,sup}_{t\geq\bar t_0}
 e^{\lambda(t-\bar t_0)}\|\mathbf w(t)\|_H^2
+\int_{\bar t_0}^\infty
 e^{\lambda(t-\bar t_0)}\|\mathbf w(t)\|_V^2\,dt.
\]
\begin{samepage}
We first show that the variational solution of
\begin{equation}\label{eq:forced-controlled}
\partial_t\mathbf w+\mathcal L(t)\mathbf w=\mathbf f,\qquad \mathbf w(\bar t_0)=\bar{\mathbf v}_0,
\end{equation}
satisfies
\begin{equation}\label{eq:weighted-linear}
\|\mathbf w\|_{X_\lambda(\bar t_0)}
\leq C_L\left(\|\bar{\mathbf v}_0\|_H+
\|e^{\lambda(\cdot-\bar t_0)/2}\mathbf f\|_{L^2(I_\infty(\bar t_0);V')}\right),
\end{equation}
whenever the right-hand side is finite. The constant $C_L$ is independent of $\bar t_0$, $\mathbf f$, and $\bar{\mathbf v}_0$.
\end{samepage}

To derive this estimate, we set $\widetilde{\mathbf w}(t)=e^{\lambda(t-\bar t_0)/2}\mathbf w(t)$ and $\widetilde{\mathbf f}(t)=e^{\lambda(t-\bar t_0)/2}\mathbf f(t)$. Differentiating this expression for $\widetilde{\mathbf w}$ and using \eqref{eq:forced-controlled}, we obtain
\begin{equation}\label{eq:shifted-controlled}
\partial_t\widetilde{\mathbf w}+\bigl(\mathcal L(t)-\tfrac\lambda2 I\bigr)\widetilde{\mathbf w}=\widetilde{\mathbf f}.
\end{equation}
For the homogeneous equation, we have $\mathbf w(t)=\mathbb W(s,t)\mathbf w(s)$ for $t\geq s\geq\bar t_0$. Substituting $\mathbf w(s)=e^{-\lambda(s-\bar t_0)/2}\widetilde{\mathbf w}(s)$ into this identity and using the linearity of $\mathbb W(s,t)$, we obtain
\[
\widetilde{\mathbf w}(t)
=e^{\lambda(t-\bar t_0)/2}\mathbb W(s,t)
  e^{-\lambda(s-\bar t_0)/2}\widetilde{\mathbf w}(s)
=e^{\lambda(t-s)/2}\mathbb W(s,t)\widetilde{\mathbf w}(s).
\]
Consequently, we use \eqref{eq:linear-feedback-bounds} to obtain
\begin{equation}\label{eq:weighted-homogeneous}
\|\widetilde{\mathbf w}(t)\|_H
\leq M e^{-(\bar\lambda-\lambda)(t-s)/2}
\|\widetilde{\mathbf w}(s)\|_H.
\end{equation}
We test \eqref{eq:shifted-controlled} with $\widetilde{\mathbf w}$ and use \eqref{eq:oseen-bounds} and \eqref{eq:linear-feedback-bounds} to obtain
\begin{equation}\label{eq:shifted-energy-testing}
\begin{aligned}
\frac12\frac{d}{dt}\|\widetilde{\mathbf w}\|_H^2
+\nu\|\widetilde{\mathbf w}\|_V^2
&\leq R\|\widetilde{\mathbf w}\|_H\|\widetilde{\mathbf w}\|_V
+\left(K_0+\frac\lambda2\right)\|\widetilde{\mathbf w}\|_H^2
+\|\widetilde{\mathbf f}\|_{V'}\|\widetilde{\mathbf w}\|_V\\
&\leq \frac\nu2\|\widetilde{\mathbf w}\|_V^2
+\left(\frac{R^2}{\nu}+K_0+\frac\lambda2\right)\|\widetilde{\mathbf w}\|_H^2
+\frac1\nu\|\widetilde{\mathbf f}\|_{V'}^2,
\end{aligned}
\end{equation}
where we apply Young's inequality to the two products in the first line. Therefore, we conclude that
\begin{equation}\label{eq:shifted-energy}
\frac{d}{dt}\|\widetilde{\mathbf w}\|_H^2+\nu\|\widetilde{\mathbf w}\|_V^2
\leq C\left(\|\widetilde{\mathbf w}\|_H^2+\|\widetilde{\mathbf f}\|_{V'}^2\right).
\end{equation}
For every fixed $\tau>0$, we first apply Gronwall's inequality to \eqref{eq:shifted-energy} to bound the $H$-norm on $[s,s+\tau]$. We then integrate the same inequality to bound the $V$-norm integral. In this way, we obtain, uniformly for $s\geq\bar t_0$,
\begin{equation}\label{eq:block-energy}
\sup_{s\leq t\leq s+\tau}\|\widetilde{\mathbf w}(t)\|_H^2+
\int_s^{s+\tau}\|\widetilde{\mathbf w}(t)\|_V^2\,dt
\leq C_\tau\left(\|\widetilde{\mathbf w}(s)\|_H^2+
\int_s^{s+\tau}\|\widetilde{\mathbf f}(t)\|_{V'}^2\,dt\right).
\end{equation}
We now choose $\tau$ sufficiently large such that
$\rho:=M e^{-(\bar\lambda-\lambda)\tau/2}<1$.
For $s_k=\bar t_0+k\tau$, we set
$a_k=\|\widetilde{\mathbf w}(s_k)\|_H$ and
$d_k=\|\widetilde{\mathbf f}\|_{L^2((s_k,s_{k+1});V')}$. On each interval $[s_k,s_{k+1}]$, we let $\widetilde{\mathbf w}_k^{\mathrm{hom}}$ and $\widetilde{\mathbf w}_k^{\mathrm{for}}$ solve
\[
\begin{aligned}
\partial_t\widetilde{\mathbf w}_k^{\mathrm{hom}}+\bigl(\mathcal L(t)-\tfrac\lambda2 I\bigr)\widetilde{\mathbf w}_k^{\mathrm{hom}}&=0,
&\widetilde{\mathbf w}_k^{\mathrm{hom}}(s_k)&=\widetilde{\mathbf w}(s_k),\\
\partial_t\widetilde{\mathbf w}_k^{\mathrm{for}}+\bigl(\mathcal L(t)-\tfrac\lambda2 I\bigr)\widetilde{\mathbf w}_k^{\mathrm{for}}&=\widetilde{\mathbf f},
&\widetilde{\mathbf w}_k^{\mathrm{for}}(s_k)&=0.
\end{aligned}
\]
By linearity and uniqueness for \eqref{eq:shifted-controlled}, we have $\widetilde{\mathbf w}=\widetilde{\mathbf w}_k^{\mathrm{hom}}+\widetilde{\mathbf w}_k^{\mathrm{for}}$ on this interval. Thus, only the forced component $\widetilde{\mathbf w}_k^{\mathrm{for}}$ has zero initial value at $s_k$. We apply \eqref{eq:weighted-homogeneous} to $\widetilde{\mathbf w}_k^{\mathrm{hom}}$ and \eqref{eq:block-energy} to $\widetilde{\mathbf w}_k^{\mathrm{for}}$ to obtain
\[
\|\widetilde{\mathbf w}_k^{\mathrm{hom}}(s_{k+1})\|_H\leq\rho a_k,\qquad
\|\widetilde{\mathbf w}_k^{\mathrm{for}}(s_{k+1})\|_H\leq C_\tau^{1/2}d_k.
\]
The triangle inequality and iteration then give
\[
a_{k+1}\leq\rho a_k+C_\tau^{1/2}d_k,\qquad
a_k\leq\rho^ka_0+C_\tau^{1/2}\sum_{j=0}^{k-1}\rho^{k-1-j}d_j.
\]

Since $0<\rho<1$, we have $\sum_{m=0}^\infty\rho^m=(1-\rho)^{-1}$. Applying the weighted Cauchy--Schwarz inequality to the sum in the preceding estimate, we obtain, for $k\geq1$,
\[
\left(\sum_{j=0}^{k-1}\rho^{k-1-j}d_j\right)^2
\leq \left(\sum_{j=0}^{k-1}\rho^{k-1-j}\right)
      \left(\sum_{j=0}^{k-1}\rho^{k-1-j}d_j^2\right)
\leq \frac1{1-\rho}\sum_{j=0}^{k-1}\rho^{k-1-j}d_j^2.
\]
We sum over $k$ and interchange the nonnegative sums to obtain
\begin{equation}\label{eq:discrete-convolution}
\sum_{k=1}^\infty\left(\sum_{j=0}^{k-1}\rho^{k-1-j}d_j\right)^2
\leq \frac1{1-\rho}\sum_{j=0}^\infty d_j^2
               \sum_{k=j+1}^\infty\rho^{k-1-j}
=\frac1{(1-\rho)^2}\sum_{j=0}^\infty d_j^2.
\end{equation}
Thus, by the triangle inequality in $\ell^2$ and $\sum_{k=0}^\infty\rho^{2k}=(1-\rho^2)^{-1}$, we obtain
\begin{equation}\label{eq:endpoint-sum}
\left(\sum_{k=0}^\infty a_k^2\right)^{1/2}
\leq \frac{a_0}{\sqrt{1-\rho^2}}
+\frac{C_\tau^{1/2}}{1-\rho}
 \left(\sum_{k=0}^\infty d_k^2\right)^{1/2}.
\end{equation}
Here $a_0=\|\bar{\mathbf v}_0\|_H$ and
$\sum_{k=0}^\infty d_k^2=\|\widetilde{\mathbf f}\|_{L^2(I_\infty(\bar t_0);V')}^2$.
We have fixed $\tau$ using only the constants in \eqref{eq:linear-feedback-bounds} and the rates $\lambda,\bar\lambda$, so the constants below remain independent of $\bar t_0$, $\mathbf f$, and $\bar{\mathbf v}_0$.

We now apply \eqref{eq:block-energy} on each interval $[s_k,s_{k+1}]$. We first sum the integral bounds and use \eqref{eq:endpoint-sum} to obtain
\begin{equation}\label{eq:global-integral-bound}
\int_{\bar t_0}^\infty\|\widetilde{\mathbf w}(t)\|_V^2\,dt
\leq C_\tau\sum_{k=0}^\infty(a_k^2+d_k^2)
\leq C\left(\|\bar{\mathbf v}_0\|_H^2+
\|\widetilde{\mathbf f}\|_{L^2(I_\infty(\bar t_0);V')}^2\right).
\end{equation}
Similarly, for every $t\in[s_k,s_{k+1}]$, \eqref{eq:block-energy} bounds $\|\widetilde{\mathbf w}(t)\|_H^2$ by $C_\tau(a_k^2+d_k^2)$. We therefore obtain
\begin{equation}\label{eq:global-supremum-bound}
\begin{aligned}
\sup_{t\geq\bar t_0}\|\widetilde{\mathbf w}(t)\|_H^2
&=\sup_{k\geq0}\sup_{s_k\leq t\leq s_{k+1}}\|\widetilde{\mathbf w}(t)\|_H^2
\leq C_\tau\sup_{k\geq0}(a_k^2+d_k^2)\\
&\leq C_\tau\sum_{k=0}^\infty(a_k^2+d_k^2)
\leq C\left(\|\bar{\mathbf v}_0\|_H^2+
\|\widetilde{\mathbf f}\|_{L^2(I_\infty(\bar t_0);V')}^2\right).
\end{aligned}
\end{equation}
Combining \eqref{eq:global-integral-bound} and \eqref{eq:global-supremum-bound}, substituting the definitions of $\widetilde{\mathbf w}$ and $\widetilde{\mathbf f}$, and taking square roots, we obtain \eqref{eq:weighted-linear} with a fixed constant $C_L$. In particular, the weighted $H$-norm bound holds for every $t\geq\bar t_0$. Finally, we use \eqref{eq:shifted-controlled}, \eqref{eq:oseen-bounds}, and \eqref{eq:linear-feedback-bounds} to obtain
$\partial_t\widetilde{\mathbf w}\in L^2(I_\infty(\bar t_0);V')$.

\medskip
\noindent\textbf{Step 2: the nonlinear estimates.}
For every $\mathbf p,\mathbf q\in X_\lambda(\bar t_0)$, we use the two-dimensional bilinear estimate \eqref{e114} and the Cauchy--Schwarz inequality in time to obtain
\begin{align*}
\|e^{\lambda(\cdot-\bar t_0)/2}B(\mathbf p,\mathbf q)\|_{L^2(V')}^2
&\leq C\int_{\bar t_0}^\infty e^{\lambda(t-\bar t_0)}
 \|\mathbf p(t)\|_H\|\mathbf q(t)\|_H
 \|\mathbf p(t)\|_V\|\mathbf q(t)\|_V\,dt\\
&\leq C\|e^{\lambda(\cdot-\bar t_0)/2}\mathbf p\|_{L^\infty(H)}
 \|e^{\lambda(\cdot-\bar t_0)/2}\mathbf q\|_{L^\infty(H)}
 \|\mathbf p\|_{L^2(V)}\|\mathbf q\|_{L^2(V)}\\
&\leq C\|\mathbf p\|_{X_\lambda(\bar t_0)}^2
 \|\mathbf q\|_{X_\lambda(\bar t_0)}^2.
\end{align*}
Here we take all time norms over $I_\infty(\bar t_0)$ and use $e^{\lambda(t-\bar t_0)/2}\geq1$ in the last step.
Consequently, we obtain the following two estimates with a fixed constant $C_N>0$ depending only on $\Omega$:
\begin{align}
\|e^{\lambda(\cdot-\bar t_0)/2}\mathcal N(\mathbf z)\|_{L^2(V')}
&\leq C_N\|\mathbf z\|_{X_\lambda(\bar t_0)}^2,\label{eq:weighted-quadratic}\\
\|e^{\lambda(\cdot-\bar t_0)/2}(\mathcal N(\mathbf z_1)-\mathcal N(\mathbf z_2))\|_{L^2(V')}
&\leq C_N(\|\mathbf z_1\|_{X_\lambda(\bar t_0)}+\|\mathbf z_2\|_{X_\lambda(\bar t_0)})
\|\mathbf z_1-\mathbf z_2\|_{X_\lambda(\bar t_0)}.\label{eq:weighted-lipschitz}
\end{align}
For the second estimate, we use the identity
$\mathcal N(\mathbf z_1)-\mathcal N(\mathbf z_2)
=B(\mathbf z_1,\mathbf z_1-\mathbf z_2)+B(\mathbf z_1-\mathbf z_2,\mathbf z_2)$.
All constants are independent of $\bar t_0$.

\medskip
\noindent\textbf{Step 3: contraction and exponential stabilization.}
For a given $\kappa>0$, we define the closed subset
\begin{equation}\label{eq:fixed-point-set}
X_\lambda^\kappa(\bar t_0)
:=\left\{\mathbf z\in X_\lambda(\bar t_0):\
\|\mathbf z\|_{X_\lambda(\bar t_0)}^2
\leq\kappa\|\bar{\mathbf v}_0\|_H^2\right\}.
\end{equation}
\begin{samepage}
We equip this set with the metric induced by $X_\lambda(\bar t_0)$, so it is complete. We choose $\kappa$ and $r_s$ below so that, whenever $\|\bar{\mathbf v}_0\|_H<r_s$, the mapping
\[
\Psi:X_\lambda^\kappa(\bar t_0)\longrightarrow X_\lambda^\kappa(\bar t_0)
\]
is well-defined and is a contraction.
\end{samepage}
For each $\mathbf z\in X_\lambda^\kappa(\bar t_0)$, we define $\Psi(\mathbf z)=\mathbf w$ by solving
\begin{equation}\label{e50}
\partial_t\mathbf w+\mathcal L(t)\mathbf w=-\mathcal N(\mathbf z),
\qquad \mathbf w(\bar t_0)=\bar{\mathbf v}_0.
\end{equation}
By \eqref{eq:weighted-quadratic}, the forcing $-\mathcal N(\mathbf z)$ has the regularity required in Step 1, so this solution exists uniquely in $X_\lambda(\bar t_0)$. We first prove that $\mathbf w$ belongs to $X_\lambda^\kappa(\bar t_0)$, which makes $\Psi$ well-defined as a map from this set into itself. We then prove contraction on the same set.

\smallskip
\noindent\emph{The mapping $\Psi:X_\lambda^\kappa(\bar t_0)\to X_\lambda^\kappa(\bar t_0)$ is well-defined.}
We fix $\mathbf z\in X_\lambda^\kappa(\bar t_0)$. Applying \eqref{eq:weighted-linear} and \eqref{eq:weighted-quadratic} to \eqref{e50}, we obtain
\begin{equation}\label{eq:fixed-point-map-bound}
\|\Psi(\mathbf z)\|_{X_\lambda(\bar t_0)}
\leq C_L\|\bar{\mathbf v}_0\|_H+C_LC_N\|\mathbf z\|_{X_\lambda(\bar t_0)}^2
\leq C_L\|\bar{\mathbf v}_0\|_H+C_LC_N\kappa\|\bar{\mathbf v}_0\|_H^2.
\end{equation}
We square \eqref{eq:fixed-point-map-bound} and use $(a+b)^2\leq2a^2+2b^2$. For $\|\bar{\mathbf v}_0\|_H<r_s$, this gives
\begin{equation}\label{eq:fixed-point-ball-bound}
\|\Psi(\mathbf z)\|_{X_\lambda(\bar t_0)}^2
\leq2C_L^2\|\bar{\mathbf v}_0\|_H^2
+2C_L^2C_N^2\kappa^2\|\bar{\mathbf v}_0\|_H^4
\leq2C_L^2\bigl(1+C_N^2\kappa^2r_s^2\bigr)\|\bar{\mathbf v}_0\|_H^2.
\end{equation}
We choose $\kappa=4C_L^2$ and $r_s\leq(C_N\kappa)^{-1}$. Then
$2C_L^2(1+C_N^2\kappa^2r_s^2)\leq\kappa$, and \eqref{eq:fixed-point-ball-bound} yields
\[
\|\Psi(\mathbf z)\|_{X_\lambda(\bar t_0)}^2
\leq\kappa\|\bar{\mathbf v}_0\|_H^2.
\]
Thus, we have $\Psi(\mathbf z)\in X_\lambda^\kappa(\bar t_0)$ for every $\mathbf z\in X_\lambda^\kappa(\bar t_0)$, and hence
$\Psi:X_\lambda^\kappa(\bar t_0)\to X_\lambda^\kappa(\bar t_0)$ is well-defined.

\smallskip
\noindent\emph{Contraction for sufficiently small $r_s$.}
We fix $\mathbf z_1,\mathbf z_2\in X_\lambda^\kappa(\bar t_0)$. Since the difference $\Psi(\mathbf z_1)-\Psi(\mathbf z_2)$ has zero initial value, we use \eqref{eq:weighted-linear} and \eqref{eq:weighted-lipschitz} to obtain
\begin{equation}\label{eq:fixed-point-difference-bound}
\|\Psi(\mathbf z_1)-\Psi(\mathbf z_2)\|_{X_\lambda(\bar t_0)}
\leq C_LC_N\bigl(\|\mathbf z_1\|_{X_\lambda(\bar t_0)}
+\|\mathbf z_2\|_{X_\lambda(\bar t_0)}\bigr)
\|\mathbf z_1-\mathbf z_2\|_{X_\lambda(\bar t_0)}.
\end{equation}
By \eqref{eq:fixed-point-set}, we have
$\|\mathbf z_i\|_{X_\lambda(\bar t_0)}\leq\sqrt\kappa\|\bar{\mathbf v}_0\|_H\leq\sqrt\kappa r_s$ for $i=1,2$. Therefore, \eqref{eq:fixed-point-difference-bound} gives
\begin{equation}\label{eq:fixed-point-ball-contraction}
\|\Psi(\mathbf z_1)-\Psi(\mathbf z_2)\|_{X_\lambda(\bar t_0)}
\leq2C_LC_N\sqrt\kappa r_s
\|\mathbf z_1-\mathbf z_2\|_{X_\lambda(\bar t_0)}.
\end{equation}
Keeping $\kappa=4C_L^2$ fixed, we reduce $r_s$ if necessary so that
$2C_LC_N\sqrt\kappa r_s\leq1/2$. For example, $r_s=(8C_L^2C_N)^{-1}$ satisfies both this condition and the condition for invariance. We then obtain a contraction of the complete set $X_\lambda^\kappa(\bar t_0)$ into itself. By the Banach fixed-point theorem, we obtain a solution $\mathbf v\in X_\lambda^\kappa(\bar t_0)$ satisfying
\[
\|\mathbf v\|_{X_\lambda(\bar t_0)}^2
\leq\kappa\|\bar{\mathbf v}_0\|_H^2
=4C_L^2\|\bar{\mathbf v}_0\|_H^2.
\]
From the equation and the above estimates, we also obtain
$\partial_t\mathbf v\in L^2(I_\infty(\bar t_0);V')$.
Thus the solution is a weak solution of \eqref{e41a}, and
\[
\|\mathbf v(t)\|_H^2\leq4C_L^2e^{-\lambda(t-\bar t_0)}
\|\bar{\mathbf v}_0\|_H^2\qquad(t\geq\bar t_0),
\]
which proves \eqref{e105} with $\Theta_1=4C_L^2$.

\medskip
\noindent\textbf{Uniqueness:} the solution of \eqref{e41a} is unique in $L^\infty(I_\infty(\bar t_0);H)\cap L^2(I_\infty(\bar t_0);V)$. Indeed, let $\mathbf a$ be another solution with the same initial state, and set $\mathbf z:=\mathbf v-\mathbf a$. Then $\mathbf z$ satisfies
\begin{equation}\label{eq:uniqueness-difference}
\begin{cases}
\partial_t\mathbf z(t)+\nu \mathcal A\mathbf z(t)+\mathcal B(\hat{\mathbf y}(t))\mathbf z(t)
+B(\mathbf z,\mathbf v)+B(\mathbf a,\mathbf z)
=\Pi\tilde{\mathbf K}_\lambda(t)\mathbf z(t),
& t\in I_\infty(\bar t_0),\\
\mathbf z(\bar t_0)=0.
\end{cases}
\end{equation}
We test \eqref{eq:uniqueness-difference} with $\mathbf z$. The second nonlinear term vanishes by \eqref{e81}. For the remaining nonlinear term, we use the antisymmetry in \eqref{e81} and the estimate \eqref{e114}. Together with \eqref{eq:oseen-bounds}, the feedback bound in \eqref{eq:linear-feedback-bounds}, and Young's inequality, this gives
\[
\frac{d}{dt}\|\mathbf z(t)\|_H^2
\leq C\bigl(1+\|\mathbf v(t)\|_V^2\bigr)\|\mathbf z(t)\|_H^2
\qquad\text{for a.e. }t\geq\bar t_0.
\]
Since $\mathbf v\in L^2(I_\infty(\bar t_0);V)$, the coefficient on the right-hand side is integrable on every finite interval. We apply Gronwall's inequality and use $\mathbf z(\bar t_0)=0$ to conclude that $\mathbf z=0$. This proves uniqueness.\qed

\end{appendices}

\end{document}